\documentclass{amsart}

\usepackage{amssymb,upref}
\usepackage[mathcal]{euscript}
\usepackage[cmtip,all]{xy}

\newcommand{\bydef}{:=}

\DeclareMathOperator{\CD}{\mathfrak{CD}}
\DeclareMathOperator{\Cl}{\mathfrak{Cl}}
\newcommand{\bi}{\mathbf{i}}

\newcommand{\sym}{\mathcal{H}}
\newcommand{\sg}{\mathrm{Sym}}

\newcommand{\cA}{\mathcal{A}}
\newcommand{\calA}{\mathcal{A}}
\newcommand{\cB}{{\mathcal B}}
\newcommand{\calB}{{\mathcal B}}
\newcommand{\cC}{\mathcal{C}}
\newcommand{\cQ}{\mathcal{Q}}
\newcommand{\cK}{\mathcal{K}}
\newcommand{\calC}{\mathcal{C}}

\newcommand{\calI}{\mathcal{I}}

\newcommand{\cL}{\mathcal{L}}

\newcommand{\cR}{\mathcal{R}}

\newcommand{\cU}{\mathcal{U}}

\newcommand{\cV}{\mathcal{V}}
\newcommand{\calV}{\mathcal{V}}
\newcommand{\calK}{{\mathcal K}}
\newcommand{\calH}{{\mathcal H}}
\newcommand{\calS}{{\mathcal S}}

\newcommand{\W}{W} 

\newcommand{\id}{{\mathrm{id}}} 

\newcommand{\frg}{{\mathfrak g}}

\newcommand{\frh}{{\mathfrak h}}

\newcommand{\frf}{{\mathfrak f}}

\newcommand{\wt}[1]{\widetilde{#1}}
\newcommand{\wb}[1]{\overline{#1}}
\newcommand{\vphi}{\varphi}

\newcommand{\espan}[1]{\mathrm{span}\left\{#1\right\}}

\DeclareMathOperator*{\ot}{\otimes}

\newcommand{\trace}{\mathrm{trace}}

\newcommand{\ZZ}{\mathbb{Z}}
\newcommand{\bZ}{{\mathbb Z}}
\newcommand{\bN}{{\mathbb N}}

\newcommand{\bR}{{\mathbb R}}
\newcommand{\bQ}{{\mathbb Q}}

\newcommand{\bC}{{\mathbb C}}

\newcommand{\bH}{{\mathbb H}}

\newcommand{\bO}{{\mathbb O}}
\newcommand{\FF}{\mathbb{F}}
\newcommand{\bF}{{\mathbb F}}
\newcommand{\bFalg}{{\overline{\bF}}}
\newcommand{\chr}[1]{\mathrm{char}\,#1}

\DeclareMathOperator{\Hom}{\mathrm{Hom}}
\DeclareMathOperator{\End}{\mathrm{End}}
\DeclareMathOperator{\Alg}{\mathrm{Alg}}

\DeclareMathOperator{\Aut}{\mathrm{Aut}}
\DeclareMathOperator{\Stab}{\mathrm{Stab}}
\DeclareMathOperator{\Diag}{\mathrm{Diag}}
\DeclareMathOperator{\AAut}{\mathbf{Aut}}
\DeclareMathOperator{\bAut}{\mathbf{Aut}}
\DeclareMathOperator{\Der}{\mathrm{Der}}

\DeclareMathOperator{\supp}{\mathrm{Supp}\,}
\DeclareMathOperator{\Supp}{\mathrm{Supp}}

\newcommand{\Lie}{\mathbf{Lie}}
\newcommand{\ad}[1]{\mathrm{ad}\,#1}
\newcommand{\Ad}[1]{\mathrm{Ad}\,#1}

\newcommand{\Gl}{\mathfrak{gl}}

\newcommand{\frsl}{{\mathfrak{sl}}}

\newcommand{\frso}{{\mathfrak{so}}}

\newcommand{\tri}{\mathfrak{tri}}

\newcommand{\GL}{\mathrm{GL}}

\newcommand{\Ort}{\mathrm{O}}
\newcommand{\SOrt}{\mathrm{SO}}

\newcommand{\Spin}{\mathrm{Spin}}

\newcommand{\Gs}{\mathbf{G}}

\newcommand{\GLs}{\mathbf{GL}}

\newcommand{\Diags}{\mathbf{Diag}}

\newcommand{\Qs}{\mathbf{Q}}
\newcommand{\Stabs}{\mathbf{Stab}}

\DeclareMathOperator{\degree}{deg}

\newcommand{\subo}{_{\bar 0}}
\newcommand{\subuno}{_{\bar 1}}

\newtheorem{theorem}{Theorem}[section]
\newtheorem{proposition}[theorem]{Proposition}
\newtheorem{lemma}[theorem]{Lemma}
\newtheorem{corollary}[theorem]{Corollary}

\theoremstyle{definition}
\newtheorem{df}[theorem]{Definition}

\theoremstyle{remark}
\newtheorem{remark}[theorem]{Remark}

\def\hregleta{\hrule height .5pt}
\def\hreglon{\hrule height1pt}
\def\vreglon{\vrule height 12pt width1pt depth 4pt}
\def\vregleta{\vrule width .5pt}
\def\hreglonfill{\leaders\hreglon\hfill}

\def\hregletafill{\leaders\hregleta\hfill}

\begin{document}

\title[Gradings on the exceptional Lie algebras $F_4$ and $G_2$ revisited]{Gradings on the exceptional Lie algebras\\ $F_4$ and $G_2$ revisited}

\author[Alberto Elduque]{Alberto Elduque$^{\star}$}
\thanks{$^{\star}$ Supported by the Spanish Ministerio de Educaci\'{o}n y Ciencia and
FEDER (MTM 2007-67884-C04-02) and by the Diputaci\'on General de Arag\'on (Grupo de Investigaci\'on de \'Algebra)}
\address{Departamento de Matem\'aticas e Instituto Universitario de Matem\'aticas y Aplicaciones,
Universidad de Zaragoza, 50009 Zaragoza, Spain}
\email{elduque@unizar.es}

\author[Mikhail Kochetov]{Mikhail Kochetov$^{\star\star}$}
\thanks{$^{\star\star}$Supported by the Natural Sciences and Engineering Research Council (NSERC) of Canada, Discovery Grant \# 341792-07.}
\address{Department of Mathematics and Statistics,
Memorial University of Newfoundland, St. John's, NL, A1C5S7, Canada}
\email{mikhail@mun.ca}




\begin{abstract}
All gradings by abelian groups are classified on the following algebras over an algebraically closed field $\FF$: the simple Lie algebra of type $G_2$ ($\chr{\FF}\ne 2,3$), the exceptional simple Jordan algebra ($\chr{\FF}\ne 2$), and the simple Lie algebra of type $F_4$ ($\chr{\FF}\ne 2$).
\end{abstract}

\maketitle


\section{Introduction}\label{se:introduction}

Gradings on Lie algebras have been extensively used since the beginning of Lie theory: the Cartan grading on a complex semisimple Lie algebra is the $\bZ^r$-grading ($r$ being the rank) whose homogeneous components are the root spaces relative to a Cartan subalgebra (which is the zero component); symmetric spaces are related to $\bZ_2$-gradings, Kac--Moody Lie algebras to gradings by a finite cyclic group, the theory of Jordan algebras and pairs to $3$-gradings on Lie algebras, etc.

In 1989, a systematic study of gradings on Lie algebras was started by Patera and Zassenhaus \cite{PateraZassenhaus}. Fine gradings (i.e., those that cannot be refined) on the classical simple complex Lie algebras other than $D_4$ by arbitrary abelian groups were considered in \cite{HPP}. The arguments there are computational and the problem of classification of fine gradings is not completely settled. The complete classification, up to equivalence, of fine gradings on all classical simple Lie algebras (including $D_4$) over algebraically closed fields of characteristic $0$ has recently been obtained in \cite{ElduqueFine}. For any abelian group $G$, the classification of all $G$-gradings, up to isomorphism, on the classical simple Lie algebras other than $D_4$ over algebraically closed fields of characteristic different from $2$ has been achieved in \cite{YuriMishaClassical} using methods developed in \cite{BSZ,BZ02,BZ03,BShZ,BZ06,BZ07,BKM}.

As to the exceptional simple Lie algebras, the classification of all gradings (up to equivalence) for type $G_2$ over an algebraically closed field of characteristic $0$ was obtained independently in \cite{CristinaCandidoG2} and \cite{BahturinTvalavadzeG2}, using the results on gradings on the Cayley algebras in \cite{ElduqueOctonions}. Also, the classification of fine gradings (up to equivalence) for type $F_4$ over an algebraically closed field of characteristic $0$ has recently been obtained in \cite{CristinaCandidoF4} (see also \cite{CristinaF4}). The method used in that work relies on the fact that, under the stated assumptions on the ground field, any abelian group grading on an algebra is  the decomposition into common eigenspaces for some diagonalizable subgroup of the automorphism group of the algebra. It is shown that any such subgroup is contained in the normalizer of a maximal torus of the automorphism group. Starting from this point, the argument is quite technical, and some computer-aided case-by-case analysis is used. Since the automorphism groups of the simple Lie algebra of type $F_4$ and of the exceptional simple Jordan algebra (the Albert algebra) are isomorphic, in \cite{CristinaCandidoF4} the fine gradings on the Albert algebra are computed as well. These methods are being currently used by C.~Draper and C.~Mart\'{\i}n-Gonz\'alez to study gradings on the simple Lie algebra of type $E_6$.

The purpose of this paper is the classification of gradings on the simple Lie algebras of types $G_2$ and $F_4$ over algebraically closed fields of characteristic different from $2$ (and different from $3$ for type $G_2$, as there is no simple Lie algebra of type $G_2$ in characteristic $3$). Actually, for $G_2$ the situation is simple enough to obtain a description of gradings without assuming the ground field algebraically closed. Our arguments will differ essentially from the arguments in \cite{CristinaCandidoG2,BahturinTvalavadzeG2,CristinaCandidoF4}, which depend heavily on the characteristic being $0$. The idea is to classify gradings on the Cayley algebra and on the Albert algebra first, and then use automorphism group schemes to transfer the classification to the corresponding Lie algebras. All gradings on the Cayley algebras over an arbitrary field were described in \cite{ElduqueOctonions}, using, essentially, only the properties of the norm and trace. All gradings on the Albert algebra over an algebraically closed field of characteristic different from $2$ will be described here, using the well-known properties of this exceptional Jordan algebra. In this way, not only the results on the gradings on the Albert algebra in \cite{CristinaCandidoF4} will be extended to positive characteristic, but also the gradings will be described intrinsically, according to structural properties of the Albert algebra and the identity component of the grading. In particular, we obtain an interesting model of the Albert algebra based on the fine $\ZZ\times\ZZ^3_2$-grading and the Cayley algebra and another model based on the fine $\ZZ_3^3$-grading and the Okubo algebra --- see \eqref{eq:nu_model} and \eqref{eq:AlbertOkubo}, respectively. Once this is done, general arguments with morphisms of affine group schemes (already used in \cite{YuriMishaClassical}) will be applied to show that any grading on the simple Lie algebra of type $G_2$ or $F_4$ is induced from a grading on the Cayley or the Albert algebra, respectively. Our desire to cover characteristic $3$ for type $F_4$ has forced us to extend some classical results which, to the best of our knowledge, have appeared in the literature only assuming characteristic different from $2$ and $3$ (see Propositions \ref{pr:innerderivations} and \ref{pr:AlbertAd}).

In Section \ref{se:gradings}, we collect the basic definitions and properties related to gradings, including their relationship with automorphism group schemes. Section \ref{se:Cayley} is devoted to a review of the description of gradings on the Cayley algebras in \cite{ElduqueOctonions} in a way suitable for our purposes; we also obtain, for any abelian group $G$, a classification of $G$-gradings up to isomorphism (over an algebraically closed field). These results are applied in Section \ref{se:G2} to describe all gradings on central simple Lie algebras of type $G_2$ over an arbitrary field of characteristic different from $2$ and $3$, and to classify the gradings up to equivalence and up to isomorphism, assuming the field algebraically closed. Then, in Section \ref{se:Albert}, the Albert algebra is described, and some subgroups of its automorphism group are considered. Section \ref{se:AlbertFine} gives constructions of four fine gradings on the Albert algebra over an algebraically closed field of characteristic different from $2$ (one of them does not exist in characteristic $3$). In Section \ref{se:classification}, these gradings are shown to exhaust the list of fine gradings, up to equivalence. We also obtain, for any abelian group $G$, a classification of $G$-gradings up to isomorphism. Finally, in Section \ref{se:F4}, all gradings on the simple Lie algebra of type $F_4$ are classified under the same assumptions on the ground field.


\section{Gradings}\label{se:gradings}

In this section, we state some basic definitions and facts concerning gradings on (nonassociative) algebras. We also fix the notation that will be used throughout the paper. The reader may consult \cite{Ksur} for a survey of results on gradings on Lie algebras.

\subsection{Some definitions}\null\quad

Let $\cA$ be an algebra over a ground field $\bF$. A \emph{grading} on $\cA$ is a decomposition
\[
\Gamma: \cA=\bigoplus_{s\in S}\cA_s
\]
of $\cA$ into a direct sum of subspaces, called the {\em homogeneous components}, such that for any $s_1,s_2\in S$ there exists $s_3\in S$ with
$\cA_{s_1}\cA_{s_2}\subset \cA_{s_3}$. If $a\in\cA_s$, we will say that $a$ is \emph{homogeneous of degree} $s$ and write $\deg a=s$.

Then:

$\bullet$ If $\cA$ is finite-dimensional, let $n_i$ be the number of homogeneous components of dimension $i$, $i=1,\ldots,r$, where $r$ is the highest dimension that occurs. (Hence $\dim \cA=\sum_{i=1}^r in_i$.) The \emph{type} of $\Gamma$ is the sequence $(n_1,n_2,\ldots,n_r)$.

$\bullet$ Two gradings $\Gamma: \cA=\bigoplus_{s\in S}\cA_s$ and $\Gamma':\cA'=\bigoplus_{s'\in S'}\cA'_{s'}$ are said to be \emph{equivalent} if there exist an isomorphism $\psi\colon\cA\rightarrow \cA'$ and a bijection $\alpha\colon S\to S'$ such that for any $s\in S$ we have $\psi(\cA_s)=\cA'_{\alpha(s)}$.

$\bullet$ Let $\Gamma$ and $\Gamma'$ be two gradings on $\cA$. The grading $\Gamma$ is said to be a \emph{refinement} of $\Gamma'$ (or $\Gamma'$ a \emph{coarsening} of $\Gamma$) if, for any $s\in S$, there exists $s'\in S'$ such that $\cA_s\subset \cA_{s'}$. In other words, each homogeneous component of $\Gamma'$ is a (direct) sum of some homogeneous components of $\Gamma$. A grading is called \emph{fine} if it admits no proper refinement.

$\bullet$ The grading $\Gamma$ is said to be a \emph{group grading} (respectively, an \emph{abelian group grading}) if there is a group (respectively, abelian group) $G$ containing $S$ such that, for all $s_1,s_2\in S$, we have $\cA_{s_1}\cA_{s_2}\subset \cA_{s_1s_2}$, with the multiplication of $s_1$ and $s_2$ in $G$. Setting $\cA_g\bydef 0$ if $g\notin S$, we have
\[
\Gamma: \cA=\bigoplus_{g\in G}\cA_g\quad\mbox{where}\quad\cA_g\cA_h\subset\cA_{gh}\quad\mbox{for all}\quad g,h\in G.
\]
This is what is called a \emph{$G$-grading} on $\cA$. A group grading (respectively, abelian group grading) is said to be \emph{fine} if it admits no proper refinement in the class of group gradings (respectively, abelian group gradings). We will also consider $G$-gradings on a vector space $V$, which are just direct sum decompositions of the form $V=\bigoplus_{g\in G}V_g$.

$\bullet$ Given a $G$-grading $\Gamma:V=\bigoplus_{g\in G}V_g$, the subset $\{g\in G\;|\;\cA_g\ne 0\}$ of $G$ will be called the \emph{support} of $\Gamma$ and denoted by $\Supp\Gamma$ (or $\Supp V$ if the grading is clear from the context). A subspace $W\subset V$ is said to be {\em graded} if $W=\bigoplus_{g\in G}W_g$ where $W_g=V_g\cap W$. Then we can speak of the support of $W$.

$\bullet$ Given a grading $\Gamma:\cA=\bigoplus_{s\in S}\cA_s$, we define the (abelian) group $G_0$ generated by $\{s\in S\;|\;\cA_s\ne 0\}$ subject only to the relations $s_1s_2=s_3$ whenever $0\ne \cA_{s_1}\cA_{s_2}\subset \cA_{s_3}$. Then we obtain a $G$-grading: $\cA=\bigoplus_{g\in G_0}\cA_g$ where $\cA_g$ is the sum of the homogeneous components $\cA_s$ such that the class of $s$ in $G$ is $g$. In general, this is a coarsening of $\Gamma$. If $\Gamma$ is a group grading (respectively, an abelian group grading), then $S$ imbeds in $G_0$ and the grading $\cA=\bigoplus_{g\in G_0}\cA_g$ coincides with $\Gamma$.
The group $G_0$ has the following universal property: given any (abelian) group grading $\cA=\bigoplus_{h\in H}\cA_h$ that is a coarsening of $\Gamma$, there exists a unique homomorphism of groups $\alpha\colon G_0\rightarrow H$  such that $\cA_h=\bigoplus_{g\in \alpha^{-1}(h)}\cA_g$. The group $G$ is called the \emph{universal (abelian) group} of $\Gamma$ and denoted $U(\Gamma)$. The universal (abelian) groups of two equivalent gradings are isomorphic.

$\bullet$ Given a $G$-grading $\Gamma:\cA=\bigoplus_{g\in G}\cA_g$ and a group homomorphism $\alpha\colon G\rightarrow H$, we obtain an $H$-grading $\cA=\bigoplus_{h\in H}\cA_h$ where  $\cA_h=\bigoplus_{g\in \alpha^{-1}(h)}\cA_g$. This $H$-grading will be denoted by ${}^\alpha\Gamma$ and said to be \emph{induced by $\alpha$ from $\Gamma$}. Clearly, ${}^\alpha\Gamma$ is coarsening of $\Gamma$ (not necessarily proper).

$\bullet$ Two $G$-gradings over the same group, $\Gamma: \cA=\bigoplus_{g\in G}\cA_g$ and $\Gamma': \cA'=\bigoplus_{g\in G}\cA'_g$, are said to be \emph{isomorphic} if there is  an isomorphism $\psi\colon\cA\rightarrow \cA'$ such that $\psi(\cA_g)=\cA'_g$ for all $g\in G$. A $G$-grading $\Gamma: \cA=\oplus_{g\in G}\cA_g$ and an $H$-grading $\Gamma': \cA'=\oplus_{h\in H}\cA'_{h}$ are said to be \emph{weakly isomorphic} if there are isomorphisms $\alpha\colon G\rightarrow H$ and $\psi\colon\cA\rightarrow\cA'$  such that, for all $g\in G$, we have $\psi(\cA_g)=\cA'_{\alpha(g)}$. This is equivalent to saying that $\Gamma'$ is isomorphic to ${}^\alpha\Gamma$.
It is clear that weakly isomorphic gradings are equivalent, but the converse does not hold in general. However, two equivalent (abelian) group gradings are weakly isomorphic when considered as gradings by their universal (abelian) groups.

$\bullet$ The {\em automorphism group} of $\Gamma$, denoted $\Aut(\Gamma)$, consists of all self-equivalences of $\Gamma$, i.e., automorphisms of $\cA$ that permute the components of $\Gamma$.
The {\em stabilizer} of $\Gamma$, denoted $\Stab(\Gamma)$,
consists of all automorphisms of the graded algebra $\cA$, i.e., automorphisms of $\cA$ that leave each component of $\Gamma$ invariant. The {\em diagonal group} of $\Gamma$, denoted $\Diag(\Gamma)$, is the subgroup of the stabilizer consisting of all automorphisms $\varphi$ such that the restriction of $\varphi$ to any homogeneous component of $\Gamma$ is the multiplication by a (nonzero) scalar. The quotient group $\Aut(\Gamma)/\Stab(\Gamma)$, which is a subgroup of $\sg(\Supp\Gamma)$, will be called the {\em Weyl group of $\Gamma$} and denoted by $\W(\Gamma)$. Each element of $\W(\Gamma)$ extends to a unique automorphism of $U(\Gamma)$, so $\W(\Gamma)$ can be regarded as a subgroup of $\Aut(U(\Gamma))$. For example, suppose $\cA$ is a finite-dimensional algebra over an algebraically closed field and $T$ is a maximal torus in the algebraic group $\Aut(\cA)$. Then the eigenspace decomposition $\Gamma$ of $\cA$ relative to $T$ is a $\mathfrak{X}(T)$-grading on $\cA$ where $\mathfrak{X}(T)$ is the group of regular characters of $T$. Let $N(T)$ be the normalizer of $T$ in $\Aut(\cA)$ and let $C(T)$ be the centralizer. It is easy to see that $T$ is the connected component of $\Diag(\Gamma)$, $\Aut(\Gamma)$ is $N(T)$, and $\Stab(\Gamma)$ is $C(T)$. Hence $\W(\Gamma)$ is $\W(T)\bydef N(T)/C(T)$. This justifies our use of the term ``Weyl group'' for $\W(\Gamma)$.

\smallskip

Unless stated otherwise, the term {\em grading} in this paper will always refer to an {\em abelian group grading}, and {\em universal group} to \emph{universal abelian group}.

\subsection{Gradings and automorphism group schemes}\null\quad

For background on group schemes the reader may consult \cite{Waterhouse} or \cite[Chapter VI]{KMRT}.

It is well-known that a $G$-grading $\Gamma$ on a vector space $V$ is equivalent to a comodule structure $\rho_\Gamma\colon V\to V\ot\FF G$, which is defined by setting $\rho_\Gamma(v)=v\ot g$ for all $v\in V_g$ and $g\in G$. Since $G$ is abelian, the Hopf algebra $\FF G$ is commutative and thus represents an affine group scheme, which we denote by $G^D$. Affine group schemes of this form are called {\em diagonalizable}. $G$ can be identified with the \emph{group of characters} of $G^D$, i.e., morphisms from $G^D$ to $\GLs_1$. If $V$ is finite-dimensional, then $\rho_\Gamma$ is equivalent to a morphism $\eta_\Gamma\colon G^D\to\GLs(V)$, i.e., a linear representation of $G^D$ on $V$. If we pick a homogeneous basis $\{v_1,\ldots,v_n\}$ in $V$, $\deg_\Gamma(v_i)=g_i$, then the comorphism of representing objects $\eta^*_\Gamma\colon\FF[X_{ij},\det(X_{ij})^{-1}]\to\FF G$ can be written explicitly as follows: $X_{ij}\mapsto\delta_{ij}g_i$, $i,j=1,\ldots,n$. In particular, $\eta_\Gamma$ is a closed imbedding if and only if $\eta^*_\Gamma$ is onto if and only if $\supp\Gamma$ generates $G$.

If $\cA$ is a finite-dimensional (nonassociative) algebra, then the \emph{automorphism group scheme} $\AAut(\cA)$ is defined as follows. For any unital commutative associative $\FF$-algebra $\cR$, the tensor product $\cA\ot\cR$ is an $\cR$-algebra, and we set
\[
\AAut(\cR)\bydef\Aut_\cR(\cA\ot\cR).
\]
Equivalently, $\AAut(\cA)$ is the subgroupscheme $\Stabs_{\GLs(\cA)}(\mu)$ where $\mu\colon\cA\ot\cA\to\cA$ is the multiplication map, which is to be regarded as an element of $\Hom(\cA\ot\cA,\cA)$ where $\GLs(\cA)$ acts in the standard way.

If $\Gamma$ is a $G$-grading on an algebra $\cA$, then the multiplication map $\mu\colon\cA\ot\cA\to\cA$ is a morphism of $G^D$-representations, which is equivalent to saying that $G^D$ stabilizes $\mu$, or that the image of $\eta_\Gamma\colon G^D\to\GLs(\cA)$ is a subgroupscheme of $\AAut(\cA)$.
Conversely, a morphism $\eta\colon G^D\to\AAut(\cA)$ gives rise to a $G$-grading $\Gamma$ on the algebra $\cA$ such that $\eta_\Gamma=\eta$.
For any unital commutative associative $\FF$-algebra $\cR$, the action of $\cR$-points of $G^D$ by automorphisms of the $\cR$-algebra $\cA\ot\cR$ can be written explicitly:
\begin{equation}\label{action_group_scheme}
(\eta_\Gamma)_\cR(f)(x\ot r)=x\ot f(g)r\quad\mbox{for all}\quad x\in\cA_g,\,r\in\cR,\,g\in G,\,f\in\Alg(\FF G,\cR).
\end{equation}

A group homomorphism $\alpha\colon G\to H$ gives rise to a morphism $\alpha^D\colon H^D\to G^D$. Then $\rho_{{}^\alpha\!\Gamma}=(\id\ot\alpha)\circ\rho_\Gamma$ implies that $\eta_{\,{}^\alpha\!\Gamma}=\eta_\Gamma\circ\alpha^D$.

Now if $\cB$ is another algebra and we have a morphism $\theta\colon\AAut(\cA)\to\AAut(\cB)$, then any $G$-grading $\Gamma$ on $\cA$ induces a $G$-grading on $\cB$ via the morphism $\theta\circ\eta_\Gamma\colon G^D\to\AAut(\cB)$. We will denote the induced grading by $\theta(\Gamma)$. Clearly, $\theta({}^\alpha\Gamma)={}^\alpha(\theta(\Gamma))$.

The group $\Aut(\cA)$ of the $\FF$-points of $\AAut(\cA)$ acts by automorphisms of $\AAut(\cA)$ via conjugation. Namely, $\vphi\in\Aut(\cA)$ defines a morphism $\Ad_\vphi\colon\AAut(\cA)\to\AAut(\cA)$ as follows:
\begin{equation}\label{action_aut_conj}
(\Ad_\vphi)_\cR(f)\bydef(\vphi\ot\id)\circ f\circ(\vphi^{-1}\ot\id)\quad\mbox{for all}\quad f\in\Aut_\cR(\cA\ot\cR).
\end{equation}
Comparing \eqref{action_group_scheme} and \eqref{action_aut_conj}, we see that $\Ad_\vphi(\Gamma)$ is the grading $\cA=\bigoplus_{g\in G}\vphi(\cA_g)$.
To summarize:

\begin{proposition}\label{duality_char_any}
The $G$-gradings on $\cA$ are in one-to-one correspondence with the morphisms of affine group schemes $G^D\to\AAut(\cA)$. Two $G$-gradings are isomorphic if and only if the corresponding morphisms are conjugate by an element of $\Aut(\cA)$. The weak isomorphism classes of gradings on $\cA$ with the property that the support generates the grading group are in one-to-one correspondence with the $\Aut(\cA)$-orbits of diagonalizable subgroupschemes in $\AAut(\cA)$.\hfill{$\square$}
\end{proposition}

Let $\Gamma$ be an abelian group grading on $\cA$. Define the subgroupscheme $\Diags(\Gamma)$ of $\AAut(\cA)$ as follows:
\[
\Diags(\Gamma)(\cR)\bydef\{f\in\Aut_\cR(\cA\ot\cR)\;|\;f|_{\cA_g\ot\cR}\in\cR^\times\id_{\cA_g\ot\cR}\mbox{ for all }g\in G\}.
\]
Since $\Diags(\Gamma)$ is a subgroupscheme of a torus in $\GLs(\cA)$, it is diagonalizable, so $\Diags(\Gamma)=U^D$ for some finitely generated abelian group $U$. If $\Gamma$ is realized as a $G$-grading, then \eqref{action_group_scheme} shows that the image of the imbedding $\eta_\Gamma\colon G^D\to\AAut(\cA)$ is a subgroupscheme of $\Diags(\Gamma)$. The imbedding $G^D\to\Diags(\Gamma)$ corresponds to an epimorphism $U\to G$. We conclude that $U$ satisfies the definition of the universal abelian group of $\Gamma$ and hence $\Diags(\Gamma)=U(\Gamma)^D$.

Let $\Gamma$ and $\Gamma'$ be two abelian group gradings on $\cA$ and let $\Qs=\Diags(\Gamma)$ and $\Qs'=\Diags(\Gamma')$. Now $\Gamma$ is a refinement of $\Gamma'$ if and only if $\Gamma'={}^\alpha\Gamma$ for some epimorphism $\alpha\colon U(\Gamma)\to U(\Gamma')$ if and only if  $\eta_{\Gamma'}=\eta_\Gamma\circ\alpha^D$. Hence we obtain
\[
\Gamma'\quad\mbox{is a coarsening of}\quad\Gamma\;\Leftrightarrow\;\Qs'\mbox{ is a subgroupscheme of }\Qs.
\]
It follows that fine gradings correspond to maximal diagonalizable subgroupschemes of $\AAut(\cA)$. To summarize:

\begin{proposition}\label{maximal_diag_gs}
The equivalence classes of fine gradings on $\cA$ are in one-to-one correspondence with the $\Aut(\cA)$-orbits of maximal diagonalizable subgroupschemes in $\AAut(\cA)$.\hfill{$\square$}
\end{proposition}

As a consequence of the descriptions in Propositions \ref{duality_char_any} and \ref{maximal_diag_gs}, we obtain the following results, which will be used to transfer the classification of gradings from the algebra of octonions to the simple Lie algebra of type $G_2$ and from the Albert algebra to the simple Lie algebra of type $F_4$.

\begin{theorem}\label{transfer}
Let $\cA$ and $\cB$ be finite-dimensional (nonassociative) algebras. Assume we have a morphism $\theta\colon\AAut(\cA)\to\AAut(\cB)$. Then, for any abelian group $G$, we have a mapping, $\Gamma\to\theta(\Gamma)$, from $G$-gradings on $\cA$ to $G$-gradings on $\cB$. If $\Gamma$ and $\Gamma'$ are isomorphic (respectively, weakly isomorphic), then $\theta(\Gamma)$ and $\theta(\Gamma')$ are isomorphic (respectively, weakly isomorphic).
\end{theorem}

\begin{proof}
We have already defined $\theta(\Gamma)$. Let $\vphi\in\Aut(\cA)$ and $\psi=\theta_\FF(\vphi)$. Then the following diagram commutes:
\[
\xymatrix{
{\AAut(\cA)}\ar[r]^\theta\ar[d]_{\Ad_\vphi} & {\AAut(\cB)}\ar[d]^{\Ad_\psi}\\
{\AAut(\cA)}\ar[r]^\theta & \AAut(\cB)
}
\]
This follows immediately from \eqref{action_aut_conj} and the equation $\theta_\cR(\vphi\ot\id)=\psi\ot\id$, which is a consequence of the naturality of $\theta$.

Now if $\vphi$ sends $\Gamma$ to $\Gamma'$ (respectively, ${}^\alpha\Gamma$ to $\Gamma'$), then $\psi$ sends $\theta(\Gamma)$ to $\theta(\Gamma')$ (respectively, $\theta({}^\alpha\Gamma)={}^\alpha(\theta(\Gamma))$ to $\theta(\Gamma')$).
\end{proof}

\begin{theorem}\label{transfer_fine}
Let $\cA$ and $\cB$ be finite-dimensional (nonassociative) algebras. Assume we have an isomorphism $\theta\colon\AAut(\cA)\to\AAut(\cB)$. Let $\Gamma$ be a $G$-grading on $\cA$ such that $G$ is its universal abelian group. Then $\Gamma$ is a fine abelian group grading if and only if so is $\theta(\Gamma)$. Also, two such fine abelian group gradings, $\Gamma$ and $\Gamma'$, are equivalent if and only if $\theta(\Gamma)$ and $\theta(\Gamma')$ are equivalent.
\end{theorem}

\begin{proof}
If $\Gamma$ is fine, then the image of $\eta_\Gamma\colon G^D\to\AAut(\cA)$ is a maximal diagonalizable subgroupscheme of $\AAut(\cA)$. Hence the image of $\eta_{\theta(\Gamma)}=\theta\circ\eta_\Gamma$ is a maximal diagonalizable subgroupscheme of $\AAut(\cB)$ and so $\theta(\Gamma)$ is fine. It remains to recall that, if universal groups are used, two fine gradings are equivalent if and only if they are weakly isomorphic, so we can apply Theorem \ref{transfer}.
\end{proof}

\subsection{Gradings on Lie algebras of derivations}\null\quad

Recall that, for any algebraic affine group scheme $\Gs$, we have the {\em adjoint representation} $\Ad\colon\Gs\rightarrow\GLs\bigl(\Lie(\Gs)\bigr)$, see e.g.
\cite[\S 21]{KMRT}.
The differential of $\Ad$ is $\ad\colon\Lie(\Gs)\rightarrow\Gl\bigl(\Lie(\Gs)\bigr)$, the adjoint representation of $\Lie(\Gs)$.
The image of $\Ad$ is contained in the subgroupscheme $\AAut(\Lie(\Gs))$ of $\GLs\bigl(\Lie(\Gs)\bigr)$, and the image of $\ad$ is contained in $\Der\bigl(\Lie(\Gs)\bigr)$.

For $\Gs=\bAut(\cA)$, we have $\Lie(\Gs)=\Der(\cA)$. Hence, given a $G$-grading $\Gamma$ on $\cA$, we get an induced $G$-grading $\Ad(\Gamma)$ on $\Der(\cA)$ by Theorem \ref{transfer}.
Since $\Ad$ in this case is the composition of the closed imbedding $\AAut(\cA)\to\GLs(\cA)$ and the standard action $\GLs(\cA)\to\GLs(\Hom(\cA,\cA))$, the grading $\Ad(\Gamma)$ is given by the standard $\FF G$-comodule structure on $\Hom(\cA,\cA)$, which is determined by the requirement that the evaluation map,
$\mathrm{ev}\colon\Hom(\cA,\cA)\ot\cA\to\cA$, be a homomorphism of $\FF G$-comodules. This implies that the induced $G$-grading $\Ad(\Gamma)$ on $\Der(\cA)$ is the natural one: $\Der(\cA)=\bigoplus_{g\in G}\Der(\cA)_g$ where
\[
\Der(\cA)_g=\{d\in\Der(\cA)\;|\;d(\cA_h)\subset\cA_{gh}\;\mbox{for all}\;h\in G\}.
\]
Let $\cL=\Der(\cA)$. If we know that $\Ad\colon\AAut(\cA)\to\AAut(\cL)$ is an isomorphism, then every $G$-grading on $\cL$ is induced from a unique $G$-grading on $\cA$ in this way, and we can transfer the classification of gradings from $\cA$ to $\cL$ via Theorems \ref{transfer} and \ref{transfer_fine}.

Let $\bFalg$ be the algebraic closure of the ground field $\FF$. In order for $\Ad$ to be an isomorphism of affine group schemes, the following conditions are necessary:
\begin{enumerate}
\item[1)] $\Ad_\bFalg\colon\Aut_\bFalg(\cA\ot\bFalg)\to\Aut_\bFalg(\cL\ot\bFalg)$ is a bijection;
\item[2)] $\ad\colon \cL\to\Der(\cL)$ is a bijection.
\end{enumerate}
If $\chr{\FF}=0$, then condition 1) alone is sufficient. If $\chr{\FF}=p$, even the combination of both conditions does not imply, in general, that $\Ad$ is an isomorphism. Recall that an algebraic affine group scheme $\Gs$ is smooth if and only if $\dim\Lie(\Gs)=\dim\Gs$ (see e.g. \cite[\S 21]{KMRT}). The dimension of $\Gs$ coincides with the dimension of the algebraic group $\Gs(\bFalg)$. Hence, for $\Gs=\AAut(\cA)$, smoothness is equivalent to the condition $\dim\Der(\cA)=\dim\Aut_\bFalg(\cA\ot\bFalg)$. If $\AAut(\cA)$ is smooth, then the combination of 1) and 2) does imply that $\Ad$ is an isomorphism of affine group schemes --- see e.g. \cite[(22.5)]{KMRT} and observe that, under conditions 1) and 2), the smoothness of $\AAut(\cA)$ implies the smoothness of $\AAut(\cL)$.


\section{Gradings on Cayley algebras}\label{se:Cayley}

The aim of this section is to present the known results about gradings on Cayley algebras in a way that will be convenient for our study of gradings on the Albert algebra.
We also obtain, for an arbitrary abelian group $G$, a classification of $G$-gradings up to isomorphism on the (unique) Cayley algebra over an algebraically closed field.
Throughout this section, the ground field $\bF$ will be arbitrary, unless stated otherwise.

A \emph{Cayley algebra} $\cC$ over $\bF$ is an eight-dimensional unital composition algebra. Then, there exists a nondegenerate quadratic form (the norm) $n\colon  \cC\rightarrow \bF$ such that $n(xy)=n(x)n(y)$ for any $x,y\in\cC$. Here the norm being nondegenerate means that its polar form: $n(x,y)=n(x+y)-n(x)-n(y)$ is a nondegenerate symmetric bilinear form.

The next result summarizes some of the well-known properties of these algebras (see \cite[Chapter VIII]{KMRT} and \cite[Chapter 2]{ZSSS}):

\begin{proposition}\label{pr:Cayley}
Let $\cC$ be a Cayley algebra over $\bF$. Then:
\begin{enumerate}
\item[1)] Any $x\in\cC$ satisfies the degree $2$ Cayley-Hamilton equation:
\begin{equation}\label{eq:CayleyHamilton}
x^2-n(x,1)x+n(x)1=0.
\end{equation}

\item[2)] The map $x\mapsto \bar x=n(x,1)1-x$ is an involution, called the \emph{standard conjugation}, of $\cC$ and for any $x,y,z\in\cC$, $x\bar x=\bar xx=n(x)1$ and $n(xy,z)=n(y,\bar xz)=n(x,z\bar y)$ hold.

\item[3)] If the norm represents $0$ --- which is always the case if $\bF$ is quadratically closed --- then there is a \emph{``good basis''} $\{e_1,e_2,u_1,u_2,u_3,v_1,v_2,v_3\}$ of $\cC$ consisting of isotropic elements, such that $n(e_1,e_2)=n(u_i,v_i)=1$ for any $i=1,2,3$ and $n(e_r,u_i)=n(e_r,v_i)=n(u_i,u_j)=n(u_i,v_j)=n(v_i,v_j)$ for any $r=1,2$ and $1\leq i\ne j\leq 3$, whose multiplication table is shown in Figure \ref{fig:Cayley}.
\begin{figure}
\[ \vbox{\offinterlineskip
\halign{\hfil$#$\enspace\hfil&#\vreglon
 &\hfil\enspace$#$\enspace\hfil
 &\hfil\enspace$#$\enspace\hfil&#\vregleta
 &\hfil\enspace$#$\enspace\hfil
 &\hfil\enspace$#$\enspace\hfil
 &\hfil\enspace$#$\enspace\hfil&#\vregleta
 &\hfil\enspace$#$\enspace\hfil
 &\hfil\enspace$#$\enspace\hfil
 &\hfil\enspace$#$\enspace\hfil&#\vreglon\cr
 &\omit\hfil\vrule width 1pt depth 4pt height 10pt
   &e_1&e_2&\omit&u_1&u_2&u_3&\omit&v_1&v_2&v_3&\omit\cr
 \noalign{\hreglon}
 e_1&&e_1&0&&u_1&u_2&u_3&&0&0&0&\cr
 e_2&&0&e_2&&0&0&0&&v_1&v_2&v_3&\cr
 &\multispan{11}{\hregletafill}\cr
 u_1&&0&u_1&&0&v_3&-v_2&&-e_1&0&0&\cr
 u_2&&0&u_2&&-v_3&0&v_1&&0&-e_1&0&\cr
 u_3&&0&u_3&&v_2&-v_1&0&&0&0&-e_1&\cr
 &\multispan{11}{\hregletafill}\cr
 v_1&&v_1&0&&-e_2&0&0&&0&u_3&-u_2&\cr
 v_2&&v_2&0&&0&-e_2&0&&-u_3&0&u_1&\cr
 v_3&&v_3&0&&0&0&-e_2&&u_2&-u_1&0&\cr
 &\multispan{12}{\hreglonfill}\cr}}
\]
\caption{Multiplication table of the Cayley algebra}\label{fig:Cayley}
\end{figure}
In particular, up to isomorphism, there is a unique Cayley algebra whose norm represents $0$, which is called the \emph{split Cayley algebra}.\qed
\end{enumerate}
\end{proposition}

A ``good basis'' $\{e_1,e_2,u_1,u_2,u_3,v_1,v_2,v_3\}$ of the split Cayley algebra $\cC$ gives a $\bZ^2$-grading with
\[
\begin{array}{lr}
\cC_{(0,0)}=\bF e_1\oplus \bF e_2,& \\
\cC_{(1,0)}=\bF u_1,& \cC_{(-1,0)}=\bF v_1,\\
\cC_{(0,1)}=\bF u_2,& \cC_{(0,-1)}=\bF v_2,\\
\cC_{(1,1)}=\bF v_3,& \cC_{(-1,-1)}=\bF u_3.
\end{array}
\]
This is called the \emph{Cartan grading} on the split Cayley algebra, and $\bZ^2$ is its universal grading group.

\begin{remark} The Cartan grading is fine as a group grading, but it is not so as a general grading, because the decomposition $C=\bF e_1\oplus\bF e_2\oplus \bF u_1\oplus\bF u_2\oplus\bF u_3\oplus \bF v_1\oplus\bF v_2\oplus\bF v_3$ is a proper refinement. This refinement is not even a semigroup grading (because $(u_1u_2)u_3=-e_2$ and $u_1(u_2u_3)=-e_1$ are in different homogeneous subspaces).
\end{remark}

\smallskip

Let $\cQ$ be a proper four-dimensional subalgebra of the Cayley algebra $\cC$ such that $n\vert_\cQ$ is nondegenerate, and let $u$ be any element in $\cC\setminus\cQ$ with $n(u)=\alpha\ne 0$. Then $\cC=\cQ\oplus\cQ u$ and we get:
\[
\begin{split}
&n(a+bu)=n(a)+\alpha n(b),\\
&(a+bu)(c+du)=(ac-\alpha\bar db) + (da+b\bar c)u,
\end{split}
\]
for any $a,b,c,d\in\cQ$. Then $\cC$ is said to be obtained from $\cQ$ by means of the \emph{Cayley--Dickson doubling process} and we write $\cC=\CD(\cQ,\alpha)$. This gives a $\bZ_2$-grading on $\cC$ with $\cC\subo=\cQ$ and $\cC\subuno=\cQ u$.

The subalgebra $\cQ$ above is a quaternion subalgebra which in turn can be obtained from a quadratic subalgebra $\cK$ through the same process $\cQ=\CD(\cK,\beta)=\cK\oplus\cK v$, and this gives a $\bZ_2$-grading of $\cQ$ and hence a $\bZ_2^2$-grading of $\cC=\cK\oplus\cK v\oplus\cK u\oplus(\cK v)u$. We write here $\cQ=\CD(\cK,\beta,\alpha)$.

If $\chr{\FF}\ne 2$, then $\cK$ can be obtained in turn from the ground field: $\cK=\CD(\bF,\gamma)$, and a $\bZ_2^3$-grading of $\cC$ appears. Here we write $\cC=\CD(\bF,\gamma,\beta,\alpha)$.

These gradings by $\bZ_2^r$, $r=1,2,3$, will be called \emph{gradings induced by the Cayley--Dickson doubling process}. The groups $\bZ_2^r$ are their universal grading groups.

\smallskip

The following result describes all possible gradings on Cayley algebras:

\begin{theorem}[\cite{ElduqueOctonions}]\label{th:grOctonions}
Any abelian group grading on a Cayley algebra is, up to equivalence, either a grading
induced by the Cayley--Dickson doubling process or a coarsening of the Cartan grading on the split Cayley algebra.\qed
\end{theorem}

\begin{remark}
The number of non-equivalent gradings induced by the Cayley--Dickson doubling process depends on the ground field. Actually, the number of non equivalent $\bZ_2$-gradings coincides with the number of isomorphism classes of quaternion subalgebras $\cQ$ of the Cayley algebra.

For an algebraically closed field $\bF$, this is one. Over $\bR$ there are two non isomorphic Cayley algebras, the classical division algebra of the octonions $\bO=\CD(\bR,-1,-1,-1)$ and the split Cayley algebra $\bO_s=\CD(\bR,1,1,1)$. Any quaternion subalgebra of $\bO$ is isomorphic to $\bH=\CD(\bR,-1,-1)$, while $\bO_s$ contains quaternion subalgebras isomorphic  to $\bH$ and to $M_2(\bR)$.

On the other hand, for $p,q$ prime numbers congruent to $3$ modulo $4$, it is easy to check that the quaternion subalgebras $\cQ_p=\CD\bigl(\bQ(\bi),p\bigr)$ and $\cQ_q=\CD\bigl(\bQ(\bi),q\bigr)$ are not isomorphic ($\bi^2=-1$). Consider the division algebra $\cQ=\CD\bigl(\bQ(\bi),-1\bigr)$. The split Cayley algebra over $\bQ$ is isomorphic to $\cC=\CD(\cQ,1)$,  and by the classical \emph{Four Squares Theorem}, $\cQ^\perp$ contains elements whose norm is $-p$ for any prime number $p$. Therefore $\cC$ contains a quaternion subalgebra isomorphic to $\cQ_p$ for any prime number $p$, and hence the split Cayley algebra over $\bQ$ is endowed with infinitely many non-equivalent $\bZ_2$-gradings.

Over an algebraically closed field there is a unique $\bZ_2^r$-grading, up to equivalence, for any $r=1,2,3$. Over $\bR$, $\bO$ is endowed with a unique $\bZ_2^r$-grading ($r=1,2,3$) up to equivalence, while $\bO_s$ is endowed with two non equivalent $\bZ_2$ and $\bZ_2^2$-gradings, but a unique $\bZ_2^3$-grading. \qed
\end{remark}

\smallskip

Up to symmetry, any coarsening of the Cartan grading is obtained as follows (with $g_i=\deg(u_i)$, $i=1,2,3$):
\begin{description}
\item[$\boxed{g_1=0}$] Then we obtain a ``$3$-grading'' by $\bZ$:
    $\cC=\cC_{-1}\oplus \cC_0\oplus \cC_1$, with $\cC_0=\espan{e_1,e_2,u_1,v_1}$,
    $\cC_1=\espan{u_2,v_3}$, $\cC_{-1}=\espan{u_3,v_2}$. All proper coarsenings have a $2$-elementary grading group.

\item[$\boxed{g_1=g_2}$] Here we obtain a ``$5$-grading'' by $\bZ$, with
    $\cC_{-2}=\bF{u_3}$, $\cC_{-1}=\espan{v_1,v_2}$, $\cC_0=\espan{e_1,e_2}$,
    $\cC_1=\espan{u_1,u_2}$ and $\cC_2=\bF{v_3}$, which has two proper coarsenings whose grading groups are not $2$-elementary:
    \begin{description}
    \item[$\boxed{g_1=g_2=g_3}$] This gives a $\bZ_3$-grading with $\cC\subo=\espan{e_1,e_2}$, $\cC\subuno=\espan{u_1,u_2,u_3}$, $\cC_{\bar 2}=\espan{v_1,v_2,v_3}$.
    \item[$\boxed{g_3=-g_3}$] This gives a $\bZ_4$-grading.
    \end{description}

\item[$\boxed{g_1=-g_1}$] Here we get a $\bZ\times\bZ_2$-grading
\[
\begin{aligned}
\cC&=\cC_{(0,\bar 0)}&&\hspace{-24pt}\oplus\ \cC_{(1,\bar 0)}
&&\hspace{-6pt}\oplus\ \cC_{(-1,\bar 0)}
&&\hspace{-6pt}\oplus\ \cC_{(0,\bar 1)}
&&\hspace{-24pt}\oplus\ \cC_{(-1,\bar 1)}
&&\hspace{-6pt}\oplus\ \cC_{(1,\bar 1)}\\[-2pt]
&\qquad\shortparallel&&\shortparallel&&\qquad \shortparallel&&\quad\ \shortparallel&&\ \shortparallel&&\qquad \shortparallel\\[-2pt]
&\espan{e_1,e_2}&&\bF{u_2}&&\qquad \bF{v_2}&&
\hspace{-4pt}\espan{u_1,v_1}&&\ \bF u_3&&\qquad \bF v_3
\end{aligned}
\]
Any of its coarsenings is a coarsening of the previous gradings.

\item[$\boxed{g_1=-g_2}$] In this case $g_3=0$, and this is equivalent to the grading obtained with $g_1=0$.
\end{description}

\smallskip

Thus the next result follows:

\begin{theorem}[\cite{ElduqueOctonions}]\label{th:Hurwitzgradings}
Up to equivalence, the nontrivial abelian group gradings on the split Cayley algebra are:
\begin{enumerate}
\item The $\bZ_2^r$-gradings induced by the Cayley--Dickson doubling process, $r=1,2,3$.
\item The Cartan grading by $\bZ^2$.
\item The $3$-grading: $\cC_0=\espan{e_1,e_2,u_3,v_3}$, $\cC_1=\espan{u_1,v_2}$, and $\cC_{-1}=\espan{u_2,v_1}$.
\item The $5$-grading: $\cC_0=\espan{e_1,e_2}$, $\cC_1=\espan{u_1,u_2}$, $\cC_2=\espan{v_3}$, $\cC_{-1}=\espan{v_1,v_2}$, and $\cC_{-2}=\espan{u_3}$.
\item The $\bZ_3$-grading: $\cC\subo=\espan{e_1,e_2}$, $\cC\subuno=\espan{u_1,u_2,u_3}$, and $\cC_{\bar 2}=\espan{v_1,v_2,v_3}$.
\item The $\bZ_4$-grading: $\cC\subo=\espan{e_1,e_2}$, $\cC\subuno=\espan{u_1,u_2}$, $\cC_{\bar 2}=\espan{u_3,v_3}$, and $\cC_{\bar 3}=\espan{v_1,v_2}$.
\item The $\bZ\times\bZ_2$-grading.\qed
\end{enumerate}
\end{theorem}

In particular, over an algebraically closed field, there are $9$ equivalence classes of nontrivial gradings on the (unique) Cayley algebra.

\begin{corollary}\label{fine_gradings_C}
Let $\Gamma$ be a fine abelian group grading on the Cayley algebra $\cC$ over an algebraically closed field $\FF$. Then $\Gamma$ is equivalent either to the Cartan grading or to the $\bZ_2^3$-grading induced by the Cayley--Dickson doubling process. The latter grading does not occur if $\chr{\FF}=2$.\qed
\end{corollary}

\smallskip

Let $G$ be an abelian group. Assuming $\FF$ algebraically closed, we can classify all $G$-gradings on $\cC$ up to isomorphism.
Let $\Gamma_\cC^1$ be the Cartan grading and let $\Gamma_\cC^2$ be the $\ZZ_2^3$-grading induced by the Cayley--Dickson doubling process.
We will need the following result:

\begin{theorem}[\cite{EK_Weyl}]
Identifying $\supp\Gamma_\cC^1$ with the short roots of the root system $\Phi$ of type $G_2$, we have $\W(\Gamma_\cC^1)=\Aut\Phi$, $\W(\Gamma_\cC^2)=\Aut(\ZZ_2^3)$.\qed
\end{theorem}

To state our classification theorem, we introduce the following notation:

$\bullet$ Let $\gamma=(g_1,g_2,g_3)$ be a triple of elements in $G$ with $g_1g_2g_3=e$. Denote by $\Gamma^1_\cC(G,\gamma)$ the $G$-grading on $\cC$ induced from $\Gamma_\cC^1$ by the homomorphism $\ZZ^2\to G$ sending $(1,0)$ to $g_1$ and $(0,1)$ to $g_2$. In other words, we set $\deg e_j=e$, $j=1,2$, $\deg u_i=g_i$ and $\deg v_i=g_i^{-1}$, $i=1,2,3$, for some ``good basis'' of $\cC$. For two such triples, $\gamma$ and $\gamma'$, we will write $\gamma\sim\gamma'$ if there exists $\pi\in\sg(3)$ such that $g'_i=g_{\pi(i)}$ for all $i=1,2,3$ or $g'_i=g^{-1}_{\pi(i)}$ for all $i=1,2,3$.

$\bullet$ Let $H\subset G$ be a subgroup isomorphic to $\ZZ_2^3$. Then $\Gamma_\cC^2$ may be regarded as a $G$-grading with support $H$. We denote this $G$-grading by $\Gamma^2_\cC(G,H)$. (Since $\W(\Gamma_\cC^2)=\Aut(\ZZ_2^3)$, all induced gradings ${}^\alpha\Gamma_\cC^2$ for various isomorphisms $\alpha\colon\ZZ_2^3\to H$ are isomorphic, so $\Gamma_\cC^2(G,H)$ is well-defined.)

\begin{theorem}\label{th:gradings_C_iso}
Let $\cC$ be the Cayley algebra over an algebraically closed field and let $G$ be an abelian group. Then any $G$-grading on $\cC$ is isomorphic to some $\Gamma^1_\cC(G,\gamma)$ or $\Gamma_\cC^2(G,H)$, but not both. Also,
\begin{itemize}
\item $\Gamma^1_\cC(G,\gamma)$ is isomorphic to $\Gamma^1_\cC(G,\gamma')$ if and only if $\gamma\sim\gamma'$;

\item $\Gamma^2_\cC(G,H)$ is isomorphic to $\Gamma^2_\cC(G,H')$ if and only if $H=H'$.
\end{itemize}
\end{theorem}

\begin{proof}
It follows from Corollary \ref{fine_gradings_C} that any $G$-grading is isomorphic to ${}^\alpha\Gamma^1_\cC$ for some $\alpha\colon\ZZ^2\to G$ or to ${}^\alpha\Gamma_\cC^2$ for some $\alpha\colon\ZZ_2^3\to G$. In the second case, if $\alpha$ is not one-to-one, then ${}^\alpha\Gamma_\cC^2$ is isomorphic to some ${}^\beta\Gamma^1_\cC$. $\Gamma^1_\cC(G,\gamma)$ and $\Gamma_\cC^2(G,T)$ cannot be isomorphic, because in the first case $\dim\cC_e\geq 2$ and in the second case $\dim\cC_e=1$.

If $\gamma\sim\gamma'$, then there is an automorphism in $\Aut(\Gamma_\cC^1)$ that sends $\Gamma^1_\cC(G,\gamma)$ to $\Gamma^1_\cC(G,\gamma')$. Conversely, if $\vphi$ is an automorphism of $\cC$ sending $\Gamma^1_\cC(G,\gamma)$ to $\Gamma^1_\cC(G,\gamma')$, then, in particular, $\vphi$ maps $\cC_e$ onto $\cC'_e$. If $\cC_e=\cC$, there is nothing to prove. Otherwise $\cC_e$ is isomorphic to $M_2(\FF)$ or $\FF\times\FF$, because it is a composition subalgebra of $\cC$ (alternatively, one may examine the cases in Theorem \ref{th:Hurwitzgradings}). If $\cC_e$ is isomorphic to $M_2(\FF)$, then one of $g_i$ is $e$. Say, $g_3=e$ and hence $g_2=g_1^{-1}$. The support of the grading then consists of $e$ and $g_1^{\pm 1}$. Applying the same argument to $g'_i$, we see that $\gamma\sim\gamma'$. Finally, consider the case $\dim\cC_e=2$. Then $\cC_e=\cC'_e$, since both are spanned by the idempotents $e_1$ and $e_2$. Hence $\vphi$ either fixes $e_1$ and $e_2$ or swaps them. In the first case, $\vphi$ preserves the subspaces $\cU$ and $\cV$. Looking at the support of $\cU$ and the dimensions of the homogeneous components in $\cU$, we conclude that $(g'_1,g'_2,g'_3)$ must be a permutation of $(g_1,g_2,g_3)$. In the second case, $\vphi$ swaps $\cU$ and $\cV$ and we conclude that $(g'_1,g'_2,g'_3)$ must be a permutation of $(g^{-1}_1,g^{-1}_2,g^{-1}_3)$.

Since $H$ is the support of $\Gamma^2_\cC(G,H)$, an isomorphism between $\Gamma^2_\cC(G,H)$ and $\Gamma^2_\cC(G,H')$ forces $H=H'$.
\end{proof}

Note that $\gamma\sim\gamma'$ if and only if the corresponding homomorphisms $\ZZ^2\to G$ are conjugate by $\W(\Gamma_\cC^1)=\Aut\Phi$ in its action on the group $U(\Gamma_\cC^1)=\ZZ^2$. This is a special case of the following general result.\footnote{The authors would like to thank Prof. Reichstein, University of British Columbia, Canada, for a discussion that was instrumental in proving this result.}

\begin{proposition}\label{pr:isomorphism_Cartan}
Let $\cA$ be a finite-dimensional algebra over an algebraically closed field. Let $T$ be a maximal torus in  $\Aut(\cA)$. Let $G$ be an abelian group and let $\Gamma$ and $\Gamma'$ be $G$-gradings induced by homomorphisms $\alpha\colon\mathfrak{X}(T)\to G$ and $\alpha'\colon\mathfrak{X}(T)\to G$, respectively. Then $\Gamma'$ is isomorphic to $\Gamma$ if and only if there exists $w\in\W(T)$ such that $\alpha'(\lambda)=\alpha(\lambda^w)$ for all $\lambda\in\mathfrak{X}(T)$.
\end{proposition}

\begin{proof}
The ``if'' part is clear. To prove the ``only if'' part, suppose $\Gamma:\cA=\bigoplus_{g\in G}\cA_g$, $\Gamma':\cA=\bigoplus_{g\in G}\cA'_g$, and there exists $\vphi\in\Aut(\cA)$ such that $\cA'_g=\vphi(\cA_g)$ for all $g\in G$. Let $T'=\vphi T \vphi^{-1}$. It is a maximal torus in $\Aut(\cA)$. Let $H=\Stab(\Gamma')$.
Then both $T$ and $T'$ are contained in $H$ and thus are maximal tori in $H$. Therefore, $T$ and $T'$ are conjugate in $H$, i.e., there exists $\psi\in H$ such that $\psi T'\psi^{-1}=T$. Let $\wt{\vphi}=\psi\vphi$. Then, by construction, we have $\wt{\vphi}T\wt{\vphi}^{-1}=T$ and $\cA'_g=\wt{\vphi}(\cA_g)$ for all $g\in G$. Hence we can take $w$ to be the image of the element $\wt{\vphi}\in N(T)$ in the quotient group $\W(T)=N(T)/C(T)$.
\end{proof}


\section{Gradings on $G_2$}\label{se:G2}

The central simple Lie algebras of type $G_2$ appear as the algebras of derivations of the Cayley algebras. The gradings on the simple Lie algebra of type $G_2$ over an algebraically closed field of characteristic $0$ were obtained independently in \cite{CristinaCandidoG2} and \cite{BahturinTvalavadzeG2}, using the results on gradings on the (unique) Cayley algebra in \cite{ElduqueOctonions}.

In this section the gradings on the simple Lie algebras of type $G_2$ will be obtained over arbitrary fields of characteristic different from $2$ and $3$. Note that, in characteristic $3$, the Lie algebra of derivations of a Cayley algebra is not simple (see e.g. \cite{CandidoElduque}).

So let $\cC$ be a Cayley algebra over a field $\bF$, $\chr{\FF}\ne 2,3$, and let $\frg=\Der(\cC)$.
Then we have the affine group scheme $\bAut(\cC)$ and the morphism $\Ad\colon\bAut(\cC)\rightarrow \bAut(\frg)$.

Let $\bFalg$ be the algebraic closure of $\bF$. Then $\Aut_\bFalg(\cC\ot\bFalg)$ is the simple algebraic group of type $G_2$. It is well-known that
\[
\Ad_\bFalg\colon\Aut_\bFalg(\cC\ot\bFalg)\to\Aut_\bFalg(\frg\ot\bFalg)
\]
is bijective. Since any derivation of $\frg$ is inner (see \cite{Seligman}), the differential
\[
\ad\colon\frg\to\Der(\frg)
\]
is also bijective. Finally, since
\[
\dim\Aut_\bFalg(\cC\ot\bFalg)=14=\dim\Der(\cC),
\]
we conclude that $\bAut(\cC)$ is smooth. It follows that $\Ad\colon\bAut(\cC)\rightarrow \bAut(\frg)$ is an isomorphism of affine group schemes and hence Theorems \ref{transfer} and \ref{transfer_fine} yield the following result:

\begin{theorem}\label{th:transfer_G2}
Let $\cC$ be a Cayley algebra over a field $\bF$, $\chr{\FF}\ne 2,3$. Then the abelian group gradings on $\Der(\cC)$ are those induced by such gradings on $\cC$.
The algebras $\cC$ and $\Der(\cC)$ have the same classification of fine gradings up to equivalence and, for any abelian group $G$, the same classification of $G$-gradings up to isomorphism.\qed
\end{theorem}

If $\cC$ is split, then $\frg=\Der(\cC)$ is the split simple Lie algebra of type $G_2$, and the Cartan grading on $\cC$ induces the Cartan decomposition of $\frg$ relative to a split Cartan subalgebra. The latter will be called the \emph{Cartan grading} on $\frg$.

\begin{corollary}
Let $\cC$ be a Cayley algebra over a field $\bF$, $\chr{\FF}\ne 2,3$. Then any abelian group grading on the simple Lie algebra $\frg=\Der(\cC)$ is, up to equivalence, either a $\bZ_2^r$-grading, $r=1,2,3$, induced by the Cayley--Dickson doubling process on $\cC$, or a coarsening of the Cartan grading on the split algebra $\frg$. In particular, if $\FF$ is algebraically closed, then there are, up to equivalence, exactly two fine abelian group gradings on $\frg$: the Cartan grading $\Gamma^1_\frg$ with universal group $\ZZ^2$ and the Cayley--Dickson grading $\Gamma^2_\frg$ with universal group $\ZZ_2^3$.\qed
\end{corollary}

Let $\Gamma^1_\frg(G,\gamma)$ and $\Gamma^2_\frg(G,H)$ be the $G$-gradings induced by $\Gamma^1_\frg$ and $\Gamma^2_\frg$, respectively, in the same way as $\Gamma^1_\cC(G,\gamma)$ and $\Gamma^2_\cC(G,H)$ are induced from $\Gamma^1_\cC$ and $\Gamma^2_\cC$ (see Theorem \ref{th:gradings_C_iso}).

\begin{corollary}\label{co:gradings_G2_iso}
Let $\frg$ be the simple Lie algebra of type $G_2$ over an algebraically closed field $\bF$, $\chr{\FF}\ne 2,3$. Let $G$ be an abelian group. Then any $G$-grading on $\frg$ is isomorphic to some $\Gamma^1_\frg(G,\gamma)$ or $\Gamma_\frg^2(G,H)$, but not both. Also,
\begin{itemize}
\item $\Gamma^1_\frg(G,\gamma)$ is isomorphic to $\Gamma^1_\frg(G,\gamma')$ if and only if $\gamma\sim\gamma'$;

\item $\Gamma^2_\frg(G,H)$ is isomorphic to $\Gamma^2_\frg(G,H')$ if and only if $H=H'$.\qed
\end{itemize}
\end{corollary}

If one wants to obtain a classification of all abelian group gradings on $\Der(\cC)$ up to equivalence, then one should be careful when applying Theorem \ref{th:transfer_G2}, because each grading on our list in Theorem \ref{th:Hurwitzgradings} can be realized as a $G$-grading for many different groups $G$.

For example, consider the $3$-grading on the split Cayley algebra $\cC$ in Theorem \ref{th:Hurwitzgradings}(3): $\cC_0=\espan{e_1,e_2,u_3,v_3}$, $\cC_1=\espan{u_1,v_2}$, $\cC_{-1}=\espan{u_2,v_1}$. As a $\bZ$-grading it induces a $5$-grading on $\Der(\cC)$, with $\Der(\cC)_2=\espan{D_{u_1,v_2}}\ne 0$, where $D_{a,b}:c\mapsto [[a,b],c]+3\bigl((ac)b-a(cb)\bigr)$ is the \emph{inner derivation} defined by $a,b\in \cC$ (the linear span of the inner derivations fills $\Der(\cC)$), so it has $5$ different nonzero homogeneous components. Its type is $(2,0,0,3)$. However, up to equivalence, this grading on $\cC$ is also a $\bZ_3$-grading, and as such it induces a $\bZ_3$-grading on $\Der(\cC)$ of type $(0,0,0,1,2)$.

As a further example, the Cartan grading on the split Cayley algebra $\cC$ can be realized as a $G$-grading for any abelian group $G$ containing two elements $g_1$ and $g_2$ such that the elements $e,g_1,g_2,g_1 g_2,g_1^{-1},g_2^{-1},(g_1 g_2)^{-1}$ are all different. In particular, it can be obtained as a $\bZ_3^2$-grading, with $g_1=(\bar 1,\bar 0)$ and $g_2=(\bar 0,\bar 1)$. However, the induced $\bZ_3^2$-grading on $\Der(\cC)$ is not equivalent to the Cartan grading, as some of the nonzero root spaces coalesce in the $\bZ_3^2$-grading.

\smallskip

Easy combinatorial arguments give all the gradings on $\Der(\cC)$ in terms of the gradings on the Cayley algebra $\cC$ in Theorem \ref{th:Hurwitzgradings} (see \cite[Figure 1]{Ksur}):

\begin{theorem}
Let $\cC$ be a split Cayley algebra over a field of characteristic different from $2$ and $3$. Up to equivalence, the nontrivial abelian group gradings on $\Der(\cC)$ are:
\begin{enumerate}
\item The $\bZ_2^r$-gradings induced by the Cayley--Dickson doubling process, $r=1,2,3$.
\item Eleven gradings induced by the Cartan grading on $\cC$ with universal groups: $\bZ^2$, $\bZ_7$, $\bZ_8$, $\bZ_9$, $\bZ_{10}$, $\bZ$, $\bZ_6\times\bZ_2$, $\bZ\times\bZ_2$, $\bZ_{12}$, $\bZ\times\bZ_3$ and $\bZ_3^2$.
\item Three gradings induced by the $3$-grading on $\cC$ with universal groups $\bZ$, $\bZ_3$ and $\bZ_4$.
\item Three gradings induced by the $5$-grading on $\cC$ with universal groups $\bZ$, $\bZ_5$ and $\bZ_6$.
\item The $\bZ_3$-grading induced by the $\bZ_3$-grading on $\cC$.
\item The $\bZ_4$-grading induced by the $\bZ_4$-grading on $\cC$.
\item Three gradings induced by the $\bZ\times\bZ_2$-grading on $\cC$ with universal groups $\bZ\times\bZ_2$, $\bZ_3\times\bZ_2$ and $\bZ_4\times\bZ_2$.\qed
\end{enumerate}
\end{theorem}

In particular, over an algebraically closed field of characteristic different from $2$ and $3$, there are exactly $25$ equivalence classes of nontrivial gradings on the  simple Lie algebra of type $G_2$.


\section{The Albert algebra}\label{se:Albert}

Let $\calC$ be the Cayley algebra over an algebraically closed field $\bF$ of characteristic different from $2$. The \emph{Albert algebra} is the algebra of Hermitian $3\times 3$-matrices over $\calC$:
\begin{equation}\label{eq:Albert}
\begin{split}
\calA=\sym_3(\calC,*)&=\left\{\begin{pmatrix} \alpha_1&\bar a_3&a_2\\ a_3&\alpha_2&\bar a_1\\ \bar a_2&a_1&\alpha_3\end{pmatrix}: \alpha_1,\alpha_2,\alpha_3\in \bF,\ a_1,a_2,a_3\in \calC\right\} \\[6pt]
&=\bF E_1\oplus\bF E_2\oplus\bF E_3\oplus \iota_1(\calC)\oplus\iota_2(\calC)\oplus\iota_3(\calC),
\end{split}
\end{equation}
where
\[
\begin{aligned}
E_1&=\begin{pmatrix}1&0&0\\ 0&0&0\\ 0&0&0\end{pmatrix}, &
E_2&=\begin{pmatrix}0&0&0\\ 0&1&0\\ 0&0&0\end{pmatrix}, &
E_3&=\begin{pmatrix}1&0&0\\ 0&0&0\\ 0&0&1\end{pmatrix}, \\
\iota_1(a)&=2\begin{pmatrix}0&0&0\\ 0&0&\bar a\\ 0&a&0\end{pmatrix},\quad &
\iota_2(a)&=2\begin{pmatrix}0&0&a\\ 0&0&0\\ \bar a&0&0\end{pmatrix},\quad &
\iota_3(a)&=2\begin{pmatrix}0&\bar a&0\\a&0&0\\ 0&0&0\end{pmatrix},\quad
\end{aligned}
\]
with (commutative) multiplication given by $X Y=\frac{1}{2}(X\cdot Y+Y\cdot X)$, where $X\cdot Y$ denotes the usual product of matrices $X$ and $Y$. Then $E_i$ are orthogonal idempotents with $E_1+E_2+E_3=1$. The rest of the products are as follows:
\begin{equation}\label{eq:Albertproduct}
\begin{split}
&E_i\iota_i(a)=0,\quad E_{i+1}\iota_i(a)=\frac{1}{2}\iota_i(a)=E_{i+2}\iota_i(a),\\
&\iota_i(a)\iota_{i+1}(b)=\iota_{i+2}(\bar a\bar b),\quad
\iota_i(a)\iota_i(b)=2n(a,b)(E_{i+1}+E_{i+2}),
\end{split}
\end{equation}
for any $a,b\in \calC$, with $i=1,2,3$ taken modulo $3$. (This convention about indices will be used without further mention.)

\smallskip

For the main properties of the Albert algebra the reader may consult \cite{Jacobson}. This is the only exceptional simple Jordan algebra over $\bF$.
Any element $X\in\calA$ satisfies the generic degree $3$ equation
\begin{equation}\label{eq:generic3}
X^3-T(X)X^2+S(X)X-N(X)1=0,
\end{equation}
for the linear form $T$ (the \emph{trace}), the quadratic form $S$, and the cubic form $N$ (the \emph{norm}) given by:
\[
\begin{split}
T(X)&=\alpha_1+\alpha_2+\alpha_3,\\
S(X)&=\frac{1}{2}\bigl(T(X)^2-T(X^2)\bigr)=\sum_{i=1}^3 \bigl(\alpha_{i+1}\alpha_{i+2}-4n(a_i)\bigr),\\
N(X)&=\alpha_1\alpha_2\alpha_3+8n(a_1,\bar a_2\bar a_3)-4\sum_{i=1}^3\alpha_in(a_i),
\end{split}
\]
for $X=\sum_{i=1}^3\bigl(\alpha_iE_i+\iota_i(a_i)\bigr)$. We note that the trace $T$ is associative:
\[
T\bigl((XY)Z\bigr)=T\bigl(X(YZ)\bigr)\quad\mbox{for all}\quad X,Y,Z\in\calA
\]
and symmetric:
\[
T(XY)=T(YX)\quad\mbox{for all}\quad X,Y\in\calA.
\]

The next result shows the good behavior of the trace form $T(X,Y)\bydef T(XY)$ of the Albert algebra with respect to gradings. It will be crucial in what follows.

\begin{theorem}\label{th:trace}
Let $G$ be an abelian group and let $\calA=\bigoplus_{g\in G}\calA_g$ be a $G$-grading on the Albert algebra over an algebraically closed field of characteristic different from $2$. Then $T(\calA_g\calA_h)=0$ unless $gh=e$.
\end{theorem}

\begin{proof}
If the characteristic of the ground field $\bF$ is not $3$, the result is very easy to prove, because $T(X)=\frac{1}{9}\trace(L_X)$ for any $X\in \calA$, where $L_X$ denotes the multiplication by $X$. Let us give a proof that includes the case of characteristic $3$. We may assume, without loss of generality, that $G$ is generated by the support of the grading, and hence it is finitely generated. It is sufficient to prove $T(\calA_g)=0$ for all $g\ne e$. If the order of $g$ is $\geq 3$, then equation \eqref{eq:generic3} shows that for any $X\in \calA_g$, $S(X)=0$ and either $T(X)=0$ or $X^2=0$. In the latter case, $T(X)^2=2S(X)+T(X^2)=0$, so again $T(X)=0$. Hence $T(\calA_g)=0$ for any $g\in G$ of order $\geq 3$. But $G=G_1G_2\cong G_1\times G_2$ where $G_2$ is the $2$-torsion subgroup of $G$ and $G_1$ is $2$-torsion free. Then $G_1$ has no elements of order $2$, and hence the trace of any non-identity homogeneous component of the $G_1$-grading induced by the projection $G\rightarrow G_1$ is $0$. In other words, $T(\calA_{gh})=0$ for any $e\ne g\in G_1$ and any $h\in G_2$. Now consider the $G_2$-grading induced by the projection $G\rightarrow G_2$. Since the characteristic is not $2$, the homogeneous components are the common eigenspaces for a family of commuting automorphisms. But for $\varphi\in\Aut(\calA)$ and $X\in \calA$ with $\varphi(X)=\lambda X$, $1\ne \lambda\in\bF$, we get $T(X)=T(\varphi(X))=\lambda T(X)$, so $T(X)=0$. Therefore, $T(\calA_{gh})=0$ for any $g\in G_1$ and $e\ne h\in G_2$. The result follows.
\end{proof}

\begin{corollary}\label{co:trace}
Under the assumtions of Theorem \ref{th:trace}, $\calA_e$ is a semisimple Jordan algebra. Moreover, if the degree of $\calA_e$ is $2$, then $\calA_e$ is isomorphic to $\bF\times\bF$.
\end{corollary}

\begin{proof}
The restriction $T\vert_{\calA_e}$ is nondegenerate by Theorem \ref{th:trace}, and if $\calI$ is an ideal of $\calA_e$ with $\calI^2=0$, then for any $X\in\calI$, $T(X\calA_e)=0$, as any element in $X\calA_e$ is nilpotent (see \cite[p.~226]{Jacobson}). Then Dieudonn\'e's Lemma \cite[p.~239]{Jacobson} proves that $\calA_e$ is semisimple.

If the degree of $\calA_e$ is $2$, then either $\calA_e$ is isomorphic to $\bF\times\bF$ (a direct sum of two copies of the degree one simple Jordan algebra), or it is a simple Jordan algebra of degree $2$. In the latter case let $\tilde m_X(\lambda)=\lambda^2-T'(X)\lambda+S'(X)$ be the generic minimal polynomial of $\calA_e$. With $m_X(\lambda)=\lambda^3-T(X)\lambda^2+S(X)\lambda-N(X)$ being the generic minimal polynomial in $\calA$, it follows that there is a linear form $T''\colon \calA_e\rightarrow \bF$ such that $m_X(\lambda)=(\lambda -T''(X))\tilde m_X(\lambda)$ for any $X\in \calA_e$ (see \cite[\S VI.3]{Jacobson}).
Then $N(X)=S'(X)T''(X)$ for any $X\in\calA_e$ and since $S'$ and $N$ are multiplicative, so is $T''$. It follows that $\ker T''$ is a codimension one ideal of $\calA_e$, a contradiction.
\end{proof}

We will make use of some subgroups of the automorphism group $\Aut(\calA)$. First we will consider $\Stab_{\Aut \calA}(E_1,E_2,E_3)$, the stabilizer of the three orthogonal idempotents $E_1$, $E_2$ and $E_3$. The orthogonal group of $\calC$ relative to its norm will be denoted by $\Ort(\calC,n)$, and the special orthogonal group by $\SOrt(\calC,n)$.

\begin{df}\label{df:relatedtriple}
A triple $(f_1,f_2,f_3)\in \Ort(\calC,n)^3$ is said to be \emph{related} if $f_1(\bar x\bar y)=\wb{f_2(x)}\,\wb{f_3(y)}$ for all $x,y\in\calC$.
\end{df}

To simplify the notation, consider the \emph{para-Hurwitz} product $x\bullet y=\bar x\bar y$ on $\calC$ --- see \cite[Chapter VIII]{KMRT}. Note that, for any $x,y,z\in \calC$, $n(x\bullet y,z)=n(\bar x\bar y,z)=n(\bar x,zy)=n(x,\bar y\bar z)=n(x,y\bullet z)$, and $(x\bullet y)\bullet x=\wb{\bar x\bar y}\bar x=(yx)\bar x=n(x)y=x\bullet (y\bullet x)$. In other words,
\begin{equation}\label{eq:paraHurwitz}
n(x\bullet y,z)=n(x,y\bullet z),\qquad (x\bullet y)\bullet x=n(x)y=x\bullet (y\bullet x),
\end{equation}
for all $x,y,z\in\calC$.

Consider the trilinear form on $\calC$ given by $\langle x,y,z\rangle=n(x\bullet y,z)$. Equation \eqref{eq:paraHurwitz} shows that $\langle x,y,z\rangle=\langle y,z,x\rangle$ for any $x,y,z\in\calC$.

\begin{lemma}\label{le:relatedtriples}
Let $f_1,f_2,f_3$ be three elements in $\Ort(\calC,n)$, then:
\begin{itemize}
\item $(f_1,f_2,f_3)$ is a related triple if and only if $\langle f_1(x),f_2(y),f_3(z)\rangle=\langle x,y,z\rangle$ for any $x,y,z\in\calC$.
\item $(f_1,f_2,f_3)$ is related if and only if so is $(f_2,f_3,f_1)$.
\end{itemize}
\end{lemma}
\begin{proof}
The triple $(f_1,f_2,f_3)$ is related if and only if $f_1(x\bullet y)=f_2(x)\bullet f_3(y)$ for any $x,y\in \calC$, and this happens if and only if $n\bigl(f_1(x\bullet y),f_1(z)\bigr)=n\bigl(f_2(x)\bullet f_3(y),f_1(z)\bigr)$ for any $x,y,z\in\calC$. But $f_1$ is orthogonal, so $n\bigl( f_1(x\bullet y),f_1(z)\bigr)=n(x\bullet y,z)$, and this is equivalent to $\langle f_2(x),f_3(y),f_1(z)\rangle =\langle x,y,z\rangle$. The cyclic symmetry of $\langle x,y,z\rangle$ completes the proof.
\end{proof}

Denote by $l_x$ and $r_x$ the left and right multiplications in the para-Cayley algebra $(\calC,\bullet)$: $l_x(y)=x\bullet y=\bar x\bar y$, $r_x(y)=y\bullet x=\bar y\bar x$. Then equation \eqref{eq:paraHurwitz} shows that $l_x^*=r_x$ and $l_xr_x=n(x)\id=r_xl_x$ for any $x\in\calC$, where $*$ denotes the adjoint relative to the norm $n$.

Let $\Cl(\cC,n)$ be the Clifford algebra of the space $\cC$ relative to the norm. The linear map
\[
\calC\longrightarrow \End_\bF(\calC\oplus\calC),\;x\mapsto \begin{pmatrix} 0&l_x\\ r_x&0\end{pmatrix},
\]
extends to an algebra isomorphism (see \cite[\S 35]{KMRT} or \cite{ElduqueTriality})
\[
\Phi\colon  \Cl(\calC,n)\rightarrow \End_\bF(\calC\oplus\calC),
\]
which is in fact an isomorphism of $\bZ_2$-graded algebras, where the Clifford algebra $\Cl(\calC,n)$ is $\bZ_2$-graded with $\degree x=\bar 1$ for all $x\in\calC$, and $\End_\bF(\calC\oplus\calC)$ is $\bZ_2$-graded with the $\bar 0$-component being the endomorphisms that preserve the two copies of $\calC$, and the $\bar 1$-component being the endomorphisms that swap these copies.

The {\em standard involution} $\tau$ on $\Cl(\calC,n)$ is defined by setting $\tau(x)=x$ for all $x\in\calC$. We define an involution on $\End_\bF(\calC\oplus\calC)$ as the adjoint relative to the quadratic form $n\perp n$ on $\cC\oplus\cC$. Since $l_x^*=r_x$ for any $x\in\calC$, it follows that $\Phi$ is an isomorphism of algebras with involution.

Consider now the corresponding spin group:
\[
\begin{split}
\Spin(\calC,n)&=\{u\in\Cl(\calC,n): u\cdot\tau(u)=1\ \text{and}\ u\cdot\calC\cdot u^{-1}\subset \calC\},\\
 &=\{x_1\cdot x_2\cdot \ldots\cdot x_{2r}: r\geq 0,\ x_i\in\calC\ \text{and}\ n(x_1)n(x_2)\cdots n(x_{2r})=1\},
\end{split}
\]
where the multiplication in $\Cl(\calC,n)$ is denoted $u\cdot v$.

For any $u\in\Spin(\calC,n)$, the map $\chi_u\colon \calC\rightarrow\calC$, $x\mapsto u\cdot x\cdot u^{-1}$ is in $\SOrt(\calC,n)$, and the map $\chi\colon \Spin(\calC,n)\rightarrow \SOrt(\calC,n)$, $u\mapsto \chi_u$ is a group homomorphism, which is onto and whose kernel is just the cyclic group of two elements $\{\pm 1\}$. Besides, for any $u\in \Spin(\calC,n)$, $\Phi(u)$ is an even endomorphism of $\calC\oplus\calC$, so there are linear maps $\rho^{\pm}_u\in\End_\bF(\calC)$ with
\[
\Phi(u)=\begin{pmatrix} \rho_u^-&0\\ 0&\rho_u^+\end{pmatrix}.
\]

\begin{theorem}\label{th:SpinRelated}
Let $\cA$ be the Albert algebra over an algebraically closed field of characteristic different from $2$. Then the map
\[
\Spin(\calC,n)\longrightarrow \GL(\calC)^3,\;u\mapsto (\chi_u,\rho_u^+,\rho_u^-),
\]
is a one-to-one group homomorphism whose image coincides with the set of related triples in $\Ort(\calC,n)^3$.
In particular, any related triple is contained in $\SOrt(\calC,n)^3$.
\end{theorem}

\begin{proof}
The map is one-to-one because so is $\Phi$. For $u\in\Spin(\calC,n)$, we have $u\cdot \tau(u)=1$, so $\rho_u^\pm\in\Ort(\calC,n)$, as $\Phi$ is an isomorphism of algebras with involution. Also, for any $x\in\calC$, $u\cdot x=\chi_u(x)\cdot u$. Applying $\Phi$ to both sides, we obtain:
\[
\begin{pmatrix} \rho^-_u&0\\ 0&\rho^+_u\end{pmatrix} \begin{pmatrix} 0&l_x\\ r_x&0\end{pmatrix}
=\begin{pmatrix} 0&l_{\chi_u(x)}\\ r_{\chi_u(x)}&0\end{pmatrix}\begin{pmatrix} \rho^-_u&0\\ 0&\rho^+_u\end{pmatrix}.
\]
Thus $\rho^-_ul_x=l_{\chi_u(x)}\rho^+_u$, or $\rho^-_u(x\bullet y)=\chi_u(x)\bullet \rho^+_u(y)$, for all $x,y\in\calC$. Hence $(\rho^-_u,\chi_u,\rho^+_u)$ is related, and so is $(\chi_u,\rho^+_u,\rho^-_u)$ by Lemma \ref{le:relatedtriples}.

Conversely, let $(f_1,f_2,f_3)$ be a related triple, and let $u$ be the (even) element in $\Cl(\calC,n)$ such that $\Phi(u)=\begin{pmatrix} f_3&0\\ 0&f_2\end{pmatrix}$. Then $u\cdot \tau(u)=1$ since $\Phi$ is an isomorphism of algebras with involution. For any $x\in\calC$,
\[
\begin{split}
\Phi(u\cdot x\cdot u^{-1})&=\begin{pmatrix}f_3&0\\ 0&f_2\end{pmatrix}
\begin{pmatrix} 0&l_x\\ r_x&0\end{pmatrix}
\begin{pmatrix} f_3^{-1}&0\\ 0&f_2^{-1}\end{pmatrix}\\
&=\begin{pmatrix} 0& f_3l_xf_2^{-1}\\ f_2r_xf_2^{-1}&0\end{pmatrix}\\
&=\begin{pmatrix} 0 & l_{f_1(x)}\\ r_{f_1(x)} & 0\end{pmatrix}\\
&=\Phi\bigl(f_1(x)\bigr),
\end{split}
\]
where we have used the equations $f_3(x\bullet y)=f_1(x)\bullet f_2(y)$ and $f_2(y\bullet x)=f_3(y)\bullet f_1(x)$).
It follows that $u\in\Spin(\calC,n)$, $\chi_u=f_1$ and hence $(f_1,f_2,f_3)=(\chi_u,\rho^+_u,\rho^-_u)$.

The last assertion follows because if $(f_1,f_2,f_3)$ is related, then there is an element $u\in\Spin(\calC,n)$ such that $f_1=\chi_u\in\SOrt(\calC,n)$. But $(f_2,f_3,f_1)$ and $(f_3,f_1,f_2)$ are also related, so $f_2,f_3\in\SOrt(\calC,n)$ as well.
\end{proof}

\begin{corollary}\label{co:StabEis}
The group $\Stab_{\Aut\calA}(E_1,E_2,E_3)$ is isomorphic to $\Spin(\calC,n)$.
\end{corollary}
\begin{proof}
Any automorphism $\varphi\in\Stab_{\Aut\calA}(E_1,E_2,E_3)$ stabilizes each of the subspaces $\iota_i(\calC)=\{X\in\calA: E_{i+1}X=\frac{1}{2}X=E_{i+2}X\}$, and hence there are linear automorphisms $f_i\in\GL(\calC)$ such that $\varphi\bigl(\iota_i(x)\bigr)=\iota_i(f_i(x))$ for any $i=1,2,3$ and $x\in\calC$. But $\iota_i(x)^2=4n(x)\bigl(E_{i+1}+E_{i+2}\bigr)$, so we obtain $f_i\in\Ort(\calC,n)$ for any $i$, and $\iota_2(x)\iota_3(y)=\iota_1(x\bullet y)$ for any $x,y\in\calC$, whence it follows that $(f_1,f_2,f_3)$ is a related triple. It remains to apply Theorem \ref{th:SpinRelated}.
\end{proof}

\begin{corollary}\label{co:StabEisiota1}
The group $\Stab_{\Aut\calA}(E_1,E_2,E_3,\iota_1(1))$ is isomorphic to $\Spin(\calC_0,n)$, where $\calC_0$ denotes the space of trace zero octonions, i.e., the orthogonal complement to $1$ in $\calC$.
\end{corollary}

\begin{proof}
Corollary \ref{co:StabEis} provides identifications:
\[
\begin{split}
\Stab_{\Aut\calA}(E_1,E_2,E_3,\iota_1(1))&\cong \{(\chi_u,\rho^+_u,\rho^-_u):u\in\Spin(\calC,n),\ \chi_u(1)=1\},\\
&\cong \{(\chi_c,\rho^+_c,\rho^-_c): c\in\Spin(\calC_0,n)\}\\
&\cong \Spin(\calC_0,n). \qedhere
\end{split}
\]
\end{proof}

Note that for $x_1,x_2\in\calC$, we have
\[
\Phi(x_1\cdot x_2)=\begin{pmatrix} 0&l_{x_1}\\ r_{x_1}&0\end{pmatrix}
\begin{pmatrix} 0&l_{x_2}\\ r_{x_2}&0\end{pmatrix}=
\begin{pmatrix} l_{x_1}r_{x_2}&0\\ 0&r_{x_1}l_{x_2}\end{pmatrix}.
\]
If $x_1,x_2\in\calC_0$, then, for any $y\in\calC$, we compute: $x_1\bullet (y\bullet x_2)=\bar x_1\wb{\bar y\bar x_2}=\bar x_1(x_2y)=-x_1(x_2y)$. Similarly, $(x_2\bullet y)\bullet x_1=-(yx_2)x_1$. Hence, for $c=x_1\cdot x_2\cdot\ldots\cdot x_{2r}\in\Spin(\calC_0,n)$, we have
\begin{equation}\label{eq:spin7}
\begin{split}
\rho^+_{x_1\cdot x_2\cdot \ldots\cdot x_{2r}}&=(-1)^rR_{x_1}R_{x_2}\cdots R_{x_{2r}},\\
\rho^-_{x_1\cdot x_2\cdot \ldots\cdot x_{2r}}&=(-1)^rL_{x_1}L_{x_2}\cdots L_{x_{2r}},
\end{split}
\end{equation}
where $L_x$ and $R_x$ denote the left and right multiplications by $x$ in $\calC$.


\section{Construction of fine gradings on the Albert algebra}\label{se:AlbertFine}

We continue to assume that the ground field $\bF$ is algebraically closed of characteristic different from $2$. The aim of this section is to construct four fine gradings on the Albert algebra (the fourth one will exist only for $\chr{\FF}\ne 3$). If $\chr{\FF}=0$, these gradings (although presented in a somewhat different form) are known to be the only fine gradings, up to equivalence \cite{CristinaCandidoG2}. The next section will be devoted to proving the same result for $\chr{\FF}\ne 2$.

\subsection{Cartan grading}\label{ss:Cartan}
Let $G=\bZ^4$ and use additive notation. Consider the following elements in $G$:
\[
\begin{aligned}
a_1&=(1,0,0,0),\quad& a_2&=(0,1,0,0),\quad&a_3&=(-1,-1,0,0),\\
g_1&=(0,0,1,0),&g_2&=(0,0,0,1),&g_3&=(0,0,-1,-1).
\end{aligned}
\]
Then $a_1+a_2+a_3=0=g_1+g_2+g_3$. Take a ``good basis'' $\{e_1,e_2,u_1,u_2,u_3,v_1,v_2,v_3\}$ of the Cayley algebra. The assignment
\[
\degree e_1=\degree e_2=0,\quad \degree u_i=g_i=-\degree v_i
\]
gives the Cartan grading of the Cayley algebra $\calC$.

Now the assignment
\[
\begin{split}
&\degree E_i=0,\\
&\degree\iota_i(e_1)=a_i=-\degree\iota_i(e_2),\\
&\degree\iota_i(u_i)=g_i=-\degree\iota_i(v_i),\\
&\degree\iota_i(u_{i+1})=a_{i+2}+g_{i+1}=-\degree\iota_i(v_{i+1}),\\
&\degree\iota_i(u_{i+2})=-a_{i+1}+g_{i+2}=-\degree\iota_i(v_{i+2}).
\end{split}
\]
for any $i=1,2,3$, gives a $\bZ^4$-grading on the Albert algebra $\calA$. Indeed, since $\calC$ is graded by the second component of $\ZZ^2\times\ZZ^2$, it suffices to look at the first component, and by the cyclic symmetry of the product, it is enough to check that $\degree\bigl(\iota_3(\bar x\bar y)\bigr)=\degree\iota_1(x)+\degree\iota_2(y)$ for any $x,y$ in the ``good basis'' of $\calC$, and this is straightforward.

This grading will be called the \emph{Cartan grading} on $\calA$. Its type is $(24,0,1)$.

Note that $\iota_i(e_1)\iota_i(e_2)=2(E_{i+1}+E_{i+2})$ is homogeneous in any refinement of the Cartan grading. Then $E_i=(E_i+E_{i+1})(E_{i-1}+E_i)$ is homogeneous too in any refinement, and it follows that $E_1,E_2,E_3$ must be homogeneous of the same degree in any refinement. Hence the Cartan grading is fine. (Actually, this proves that it is fine not just as an abelian group grading, but as a general grading.)

Also, the elements
\begin{equation}\label{eq:generatorsCartan}
\iota_1(e_1),\ \iota_1(e_2),\, \iota_2(e_1),\, \iota_2(e_2),\, \iota_1(u_1),\, \iota_1(v_1),\,\iota_2(u_2),\,\iota_2(v_2)
\end{equation}
constitute a set of generators of $\calA$. In any grading $\Gamma: \calA=\bigoplus_{g\in G}\calA_g$ in which these elements are homogeneous, as $\iota_1(e_1)\iota_1(e_2)=2(E_2+E_3)$, we obtain that $E_2+E_3$ is homogeneous. But this is an idempotent, so its degree must be $e$, and we have $\degree\iota_1(e_1)\degree\iota_1(e_2)=e$. In the same vein, $\degree\iota_2(e_1)\degree\iota_2(e_2)=\degree\iota_1(u_1)\degree\iota_1(v_1)=\degree\iota_2(u_2)\degree\iota_2(v_2)=e$. Therefore the assignment $a_1\mapsto \degree\iota_1(e_1)$, $a_2\mapsto \degree\iota_2(e_1)$, $g_1\mapsto \degree\iota_1(u_1)$ and $g_2\mapsto \degree\iota_2(u_2)$ determines a group homomorphism $\alpha\colon \bZ^4\rightarrow G$.

This proves the following result:

\begin{theorem}\label{th:CartanGrading}
Let $\Gamma: \calA=\bigoplus_{g\in G}\calA_g$ be a grading of the Albert algebra in which the elements in \eqref{eq:generatorsCartan} are homogeneous. Then there is a group homomorphism $\alpha\colon \bZ^4\rightarrow G$ such that $\Gamma$ is the grading induced by $\alpha$ from the Cartan grading.

In particular, $\bZ^4$ is the universal group of the Cartan grading. \qed
\end{theorem}

\subsection{$\bZ_2^5$-grading}\label{ss:Z25}

As discussed in Section \ref{se:Cayley}, the Cayley algebra $\calC$ is obtained by repeated application of the Cayley--Dickson doubling process:
\[
\calK=\bF\oplus\bF w_1,\quad \calH=\calK\oplus\calK w_2,\quad \calC=\calH\oplus\calH w_3,
\]
with $w_i^2=1$ for $i=1,2,3$ (one may take $w_1=e_1-e_2$, $w_2=u_1-v_1$ and $w_3=u_2-v_2$), and this gives a (uniquely determined up to isomorphism) $\bZ_2^3$-grading of $\calC$ by setting $\degree w_1=(\bar 1,\bar 0,\bar 0)$, $\degree w_2=(\bar 0,\bar 1,\bar 0)$, $\degree w_3=(\bar 0,\bar 0,\bar 1)$.

Then $\calA$ is obviously $\bZ_2^5$-graded as follows:
\[
\begin{split}
\degree E_i&=(\bar 0,\bar 0,\bar 0,\bar 0,\bar 0),\ i=1,2,3\\
\degree\iota_1(x)&=(\bar 1,\bar 0,\degree x),\\
\degree\iota_2(x)&=(\bar 0,\bar 1,\degree x),\\
\degree\iota_3(x)&=(\bar 1,\bar 1,\degree x),
\end{split}
\]
for homogeneous elements $x\in \calC$. The type of this grading is $(24,0,1)$.

This grading will be referred to as the \emph{$\bZ_2^5$-grading} on $\cA$.

With the same arguments as for the Cartan grading, this grading is fine (even as a general grading).

\begin{theorem}\label{th:Z25Grading}
Let $\Gamma: \calA=\bigoplus_{g\in G}\calA_g$ be a grading of the Albert algebra in which the elements
\[
\iota_1(1),\,\iota_2(1),\, \iota_3(w_j),\; j=1,2,3,
\]
are homogeneous. Then there is a group homomorphism $\alpha\colon \bZ_2^5\rightarrow G$ such that $\Gamma$ is the grading induced by $\alpha$ from the $\bZ_2^5$-grading.

In particular $\bZ_2^5$ is the universal group of the $\bZ_2^5$-grading.
\end{theorem}

\begin{proof}
Since $\iota_1(1)$ is homogeneous for $\Gamma$, so is $\iota_1(1)^2=4(E_2+E_3)$. But $E_2+E_3$ is an idempotent, so its degree must be $e$, and hence the degree of $\iota_1(1)$ has order $\le 2$. The same happens to all the homogeneous elements above, and since these elements constitute a set of generators of $\calA$, the result follows.
\end{proof}

\subsection{$\bZ\times\bZ_2^3$-grading}\label{ss:ZZ23}

Take an element $\bi\in\bF$ with $\bi^2=-1$ and consider the following elements in $\calA$:
\[
\begin{split}
E&=E_1,\\
\wt{E}&=1-E=E_2+E_3,\\
\nu(a)&=\bi\iota_1(a)\quad\text{for all}\quad a\in\calC_0,\\
\nu_{\pm}(x)&=\iota_2(x)\pm \bi\iota_3(\bar x)\quad\text{for all}\quad x\in \calC,\\
S^{\pm}&=E_3-E_2\pm\frac{\bi}{2}\iota_1(1).
\end{split}
\]
These elements span $\calA$, and the multiplication is given by:
\begin{equation}\label{eq:nu_model}
\begin{split}
&E\wt{E}=0,\quad ES^{\pm}=0,\quad E\nu(a)=0,\quad E\nu_{\pm}(x)=\frac{1}{2}\nu_{\pm}(x),\\
&\wt{E} S^{\pm}=S^{\pm},\quad \wt{E}\nu(a)=\nu(a),\quad \wt{E}\nu_{\pm}(x)=\frac{1}{2}\nu_{\pm}(x),\\
&S^+S^-=2\wt{E},\quad S^{\pm}\nu(a)=0,\quad S^{\pm}\nu_{\mp}(x)=\nu_{\pm}(x),\quad S^{\pm}\nu_{\pm}(x)=0,\\
&\nu(a)\nu(b)=-2n(a,b)\wt{E},\quad \nu(a)\nu_{\pm}(x)=\pm\nu_{\pm}(xa),\\
&\nu_{\pm}(x)\nu_{\pm}(y)=2n(x,y)S^{\pm},\quad\nu_+(x)\nu_-(y)=2n(x,y)(2E+\wt{E})+\nu(\bar xy-\bar yx),
\end{split}
\end{equation}
for any $x,y\in\calC$ and $a,b\in \calC_0$.

There appears a $\bZ$-grading on $\calA$:
\begin{equation}\label{eq:ZGrading}
\calA=\calA_{-2}\oplus\calA_{-1}\oplus\calA_0\oplus\calA_1\oplus\calA_2,
\end{equation}
with $\calA_{\pm 2}=\bF S^{\pm}$, $\calA_{\pm 1}=\nu_{\pm}(\calC)$, and $\calA_0=\bF E\oplus\Bigl(\bF \wt{E}\oplus \nu(\calC_0)\Bigr)$.
Note that the subspace $\bF\wt{E}\oplus\nu(\calC_0)$ is the Jordan algebra of the quadratic form $-4n\vert_{\calC_0}$, with unity $\wt{E}$.

The $\bZ_2^3$-grading on $\calC$ considered previously combines with this $\bZ$-grading to give a $\bZ\times\bZ_2^3$-grading as follows:
\[
\begin{split}
\degree S^\pm&=(\pm 2,\bar 0,\bar 0,\bar 0),\\
\degree\nu_{\pm}(x)&=(\pm 1,\degree x),\\
\degree E&=0=\degree\wt{E},\\
\degree\nu(a)&=(0,\degree a),
\end{split}
\]
for homogeneous elements $x\in\calC$ and $a\in\calC_0$.

This grading will be referred to as the \emph{$\bZ\times\bZ_2^3$-grading} on $\cA$. Its type is $(25,1)$ and again it is fine (even as a general grading).

\begin{theorem}\label{th:ZZ23Grading}
Let $\Gamma: \calA=\bigoplus_{g\in G}\calA_g$ be a grading of the Albert algebra in which the elements
\[
\nu_{\pm}(1),\quad \nu(w_j),\; j=1,2,3,
\]
are homogeneous. Then there is a group homomorphism $\alpha\colon \bZ\times \bZ_2^3\rightarrow G$ such that $\Gamma$ is the grading induced by $\alpha$ from the $\bZ\times\bZ_2^3$-grading.

In particular $\bZ\times \bZ_2^3$ is the universal group of the $\bZ\times\bZ_2^3$-grading.
\end{theorem}
\begin{proof}
As in Theorem \ref{th:Z25Grading}, if $\nu(w_j)$ is homogeneous for $\Gamma$, then its degree has order $\le 2$, and as in Theorem \ref{th:CartanGrading}, if $\nu_{\pm}(1)$ is homogeneous, then $\degree\nu_+(1)\degree\nu_-(1)=e$. Since the elements above constitute a set of generators of $\calA$, the result follows.
\end{proof}

\begin{remark}\label{re:StabEiiota1}
Note that the stabilizer $\Stab_{\Aut\calA}(E_1,E_2,E_3,\iota_1(1))$, which is isomorphic to $\Spin(\calC_0,n)$ by Corollary \ref{co:StabEisiota1}, coincides with $\Stab_{\Aut\calA}(E,S^+,S^-)$. Also, relative to the $\bZ$-grading in equation \eqref{eq:ZGrading}:
\[
\calA_{\pm 1}=\{X\in\calA\;|\; S^{\mp}X=0,\, EX=\frac{1}{2}X\},\
\nu(\calC_0)=\{X\in\calA\;|\; S^{\pm}X=0=EX\}.
\]
Hence $\Stab_{\Aut \calA}(E_1,E_2,E_3,\iota_1(1))$ stabilizes the $\bZ$-grading. Moreover, given any $c=x_1\cdot x_2\cdot\ldots\cdot x_{2r}\in\Spin(\calC_0,n)$, i.e.,  $x_j\in \calC_0$ for any $j$ and $n(x_1)n(x_2)\cdots n(x_{2r})=1$, the corresponding automorphism $\varphi_c$ in $\Stab_{\Aut \calA}(E_1,E_2,E_3,\iota_1(1))$ fixes $E_i$, $i=1,2,3$, acts as $\chi_c$ on $\iota_1(\calC)$, as $\rho^+_c=(-1)^rR_{x_1}R_{x_2}\cdots R_{x_{2r}}$ on $\iota_2(\calC)$ and as $\rho^-_c=(-1)^rL_{x_1}L_{x_2}\cdots L_{x_{2r}}$ on $\iota_3(\calC)$ --- see \eqref{eq:spin7}. But $\nu_{\pm}(x)=\iota_2(x)\pm \bi\iota_3(\bar x)$, so for all $x\in \calC$, we have:
\[
\begin{split}
\varphi_c(\nu_{\pm}(x))&=(-1)^r\Bigl(\iota_2\bigl(((xx_{2r})\cdots)x_1\bigr)\pm \bi\iota_3\bigl(x_1(\cdots(x_{2r}\bar x))\bigr)\Bigr)\\
&=(-1)^r\Bigl(\iota_2\bigl(((xx_{2r})\cdots)x_1\bigr)\pm \bi\iota_3\bigl(\wb{((xx_{2r})\cdots)x_1}\bigr)\Bigr)\\
&=\nu_{\pm}(\rho^+_c(x)).
\end{split}
\]
\end{remark}

\subsection{$\bZ_3^3$-grading}\label{ss:Z33}

Define an order $3$ automorphism $\tau$ of $\calC$ that acts on the elements of a ``good basis'' of $\calC$ as follows:
\[
\tau(e_i)=e_i,\quad \tau(u_j)=u_{j+1},\quad \tau(v_j)=v_{j+1}
\]
for $i=1,2$ and $j=1,2,3$, and a new multiplication on $\calC$:
\[
x*y=\tau(\bar x)\tau^2(\bar y),
\]
for all $x,y\in \calC$. Then $n(x*y)=n(x)n(y)$ for any $x,y$, since $\tau$ preserves the norm. Moreover, for any $x,y,z\in \calC$:
\[
\begin{split}
n(x*y,z)&=n(\tau(\bar x)\tau^2(\bar y),z)\\
    &=n(\tau(\bar x),z\tau^2(y))\\
    &=n(\bar x,\tau^2(z)\tau(y))\\
    &=n(x,\tau(\bar y)\tau^2(\bar z))\\
    &=n(x,y*z).
\end{split}
\]
Hence $(\calC,*,n)$ is a symmetric composition algebra (see \cite{ElduqueGrSym} or \cite[Chapter VIII]{KMRT}). Actually, $(\calC,*)$ is the Okubo algebra over $\bF$. Its multiplication table is shown in Figure \ref{fig:Okubo}.
\begin{figure}
\[
\vbox{\offinterlineskip
\halign{\hfil$#$\enspace\hfil&#\vreglon
 &\hfil\enspace$#$\enspace\hfil
 &\hfil\enspace$#$\enspace\hfil&#\vregleta
 &\hfil\enspace$#$\enspace\hfil
 &\hfil\enspace$#$\enspace\hfil&#\vregleta
 &\hfil\enspace$#$\enspace\hfil
 &\hfil\enspace$#$\enspace\hfil&#\vregleta
 &\hfil\enspace$#$\enspace\hfil
 &\hfil\enspace$#$\enspace\hfil&#\vreglon\cr
 &\omit\hfil\vrule width 1pt depth 4pt height 10pt
   &e_1&e_2&\omit&u_1&v_1&\omit&u_2&v_2&\omit&u_3&v_3&\omit\cr
 \noalign{\hreglon}
 e_1&&e_2&0&&0&-v_3&&0&-v_1&&0&-v_2&\cr
 e_2&&0&e_1&&-u_3&0&&-u_1&0&&-u_2&0&\cr
 &\multispan{12}{\hregletafill}\cr
 u_1&&-u_2&0&&v_1&0&&-v_3&0&&0&-e_1&\cr
 v_1&&0&-v_2&&0&u_1&&0&-u_3&&-e_2&0&\cr
 &\multispan{12}{\hregletafill}\cr
 u_2&&-u_3&0&&0&-e_1&&v_2&0&&-v_1&0&\cr
 v_2&&0&-v_3&&-e_2&0&&0&u_2&&0&-u_1&\cr
 &\multispan{12}{\hregletafill}\cr
 u_3&&-u_1&0&&-v_2&0&&0&-e_1&&v_3&0&\cr
 v_3&&0&-v_1&&0&-u_2&&-e_2&0&&0&u_3&\cr
 &\multispan{13}{\hreglonfill}\cr}}
\]
\caption{Multiplication table of the Okubo algebra}\label{fig:Okubo}
\end{figure}

This Okubo algebra is $\bZ_3^2$-graded by setting $\degree e_1=(\bar 1,\bar 0)$ and $\degree u_1=(\bar 0,\bar 1)$, with the degrees of the remaining elements being uniquely determined.

Assume now that $\chr{\bF}\ne 3$. Then this $\bZ_3^2$-grading is determined by two commuting order $3$ automorphisms $\varphi_1,\varphi_2\in\Aut(\calC,*)$:
\[
\begin{aligned}
\varphi_1(e_1)&=\omega e_1,\quad &\varphi_1(u_1)&=u_1,\\
\varphi_2(e_1)&=e_1,\quad &\varphi_2(u_1)&=\omega u_1,
\end{aligned}
\]
where $\omega$ is a primitive third root of unity in $\bF$.

Define now $\tilde\iota_i(x)=\iota_i(\tau^i(x))$ for all $i=1,2,3$ and $x\in\calC$. Then the multiplication in the Albert algebra $\calA=\oplus_{i=1}^3\bigl(\bF E_i\oplus \tilde\iota_i(\calC)\bigr)$ is given by:
\begin{equation}\label{eq:AlbertOkubo}
\begin{split}
&E_i^2=E_i,\quad E_iE_{i+1}=0,\\
&E_i\tilde\iota_i(x)=0,\quad E_{i+1}\tilde\iota_i(x)=\frac{1}{2}\tilde\iota_i(x)=E_{i+2}\tilde\iota_i(x),\\
&\tilde\iota_i(x)\tilde\iota_{i+1}(y)=\tilde\iota_{i+2}(x*y),\quad
\tilde\iota_i(x)\tilde\iota_i(y)=2n(x,y)(E_{i+1}+E_{i+2}),
\end{split}
\end{equation}
for $i=1,2,3$ and $x,y\in\calC$.

The commuting order $3$ automorphisms $\varphi_1$, $\varphi_2$ of $(\calC,*)$ extend to commuting order $3$ automorphisms of $\calA$ (which will be denoted by the same symbols) as follows: $\varphi_j(E_i)=E_i$, $\varphi_j\bigl(\tilde\iota_i(x)\bigr)=\tilde\iota_i(\varphi_j(x))$ for all $i=1,2,3$, $j=1,2$ and $x\in \calC$. On the other hand, the linear map $\varphi_3\in\End_\bF(\calA)$ defined by
\[
\varphi_3(E_i)=E_{i+1},\quad \varphi_3\bigl(\tilde\iota_i(x)\bigr)=\tilde\iota_{i+1}(x),
\]
for all $i=1,2,3$ and $x\in\calC$, is another order $3$ automorphism, which commutes with $\varphi_1$ and $\varphi_2$. The subgroup of $\Aut(\calA)$ generated by $\varphi_1,\varphi_2,\varphi_3$ is isomorphic to $\bZ_3^3$ and induces a $\bZ_3^3$-grading on $\calA$ of type $(27)$. This grading is obviously fine, and $\bZ_3^3$ is its universal group.

This grading will be referred to as the \emph{$\bZ_3^3$-grading} on $\cA$ ($\chr{\FF}\ne 3$).

\begin{remark}\label{re:TitsConstruction}
We may define the elements
\[
\begin{split}
\rho_{\bar 0}(x)&=\tilde\iota_1(x)+\tilde\iota_2(x)+\tilde\iota_3(x),\\
\rho_{\bar 1}(x)&=\tilde\iota_1(x)+\omega^2\tilde\iota_2(x)+\omega\tilde\iota_3(x),\\
\rho_{\bar 2}(x)&=\tilde\iota_1(x)+\omega\tilde\iota_2(x)+\omega^2\tilde\iota_3(x),
\end{split}
\]
for any $x\in \calC$. Then the eigenspaces of $\varphi_3$ are:
\[
\begin{split}
\calA_{\bar 0}&=\bF 1\oplus \rho_{\bar 0}(\calC)\quad (1=E_1+E_2+E_3),\\
\calA_{\bar 1}&=\bF(E_1+\omega^2E_2+\omega E_3)\oplus \rho_{\bar 1}(\calC),\\
\calA_{\bar 2}&=\bF(E_1+\omega E_2+\omega^2 E_3)\oplus \rho_{\bar 2}(\calC).
\end{split}
\]
The subalgebra $\calA_{\bar 0}$ is isomorphic to the Jordan algebra $M_3(\bF)^+$, the $3\times 3$ matrices with the symmetrized product, and the decomposition $\calA=\calA_{\bar 0}\oplus\calA_{\bar 1}\oplus\calA_{\bar 2}$ gives the \emph{First Tits Construction} of $\calA$ (see \cite[p.~412]{Jacobson}).
\end{remark}


\section{Classification of gradings on the Albert algebra}\label{se:classification}

The aim of this section is to classify the fine gradings on the Albert algebra $\cA$ up to equivalence and then, for any abelian group $G$, all $G$-gradings on $\cA$ up to isomorphism. Throughout this section, we will assume that the ground field $\FF$ is algebraically closed of characteristic different from $2$.

\begin{theorem}\label{th:Main}
Let $\cA$ be the Albert algebra over an algebraically closed field $\FF$, $\chr{F}\ne 2$. Then, up to equivalence, the fine abelian group gradings on $\cA$, their universal groups and types are the following:
\begin{itemize}
\item The Cartan grading $\Gamma_\cA^1$ defined in \S\ref{ss:Cartan}; universal group $\bZ^4$; type $(24,0,1)$.

\item The grading $\Gamma_\cA^2$ defined in \S\ref{ss:Z25}; universal group $\bZ_2^5$; type $(24,0,1)$.

\item The grading $\Gamma_\cA^3$ defined in \S\ref{ss:ZZ23}; universal group $\bZ\times \bZ_2^3$; type $(25,1)$.

\item If $\chr{\FF}\ne 3$, then also the grading $\Gamma_\cA^4$ defined in \S\ref{ss:Z33}; universal group $\bZ_3^3$; type $(27)$.
\end{itemize}
\end{theorem}

We already know that the gradings $\Gamma^j_\cA$, $j=1,2,3,4$, are fine (Theorems \ref{th:CartanGrading}, \ref{th:Z25Grading} and \ref{th:ZZ23Grading} for $j=1,2,3$, obvious for $\Gamma_\cA^4$). So it will suffice to show that any grading $\Gamma: \calA=\bigoplus_{g\in G}\calA_g$ of the Albert algebra is induced from $\Gamma^j_\cA$ for some $j=1,2,3,4$ ($j\ne 4$ if $\chr{\FF}=3$), by a homomorphism $U(\Gamma^j_\cA)\to G$.
The proof will be divided into cases according to the degree of the semisimple subalgebra $\calA_e$, which can be $1$, $2$ or $3$ (see Corollary \ref{co:trace}).

\subsection{Degree $3$}\label{ss:Degree3}

In case the degree of $\calA_e$ is $3$, $\calA_e$ contains three orthogonal primitive idempotents, and the coordinatization results in \cite[\S III.2 and \S IX.1]{Jacobson} show that we may assume that $E_1,E_2,E_3$ are in $\calA_e$. Hence the subspaces $\iota_i(\calC)=\{X\in \calA: E_{i+1}X=E_{i+2}X=\frac{1}{2}X\}$ are graded subspaces of $\calA$, $i=1,2,3$.

Assume first that for some $i$ there is a basis of $\iota_i(\calC)$ consisting of homogeneous elements: $\{\iota_i(x_j),\iota_i(y_j): j=1,2,3,4\}$  such that $n(x_j,y_k)=\delta_{ij}$, $n(x_j,x_k)=0=n(y_j,y_k)$ (a basis consisting of four orthogonal hyperbolic pairs). This is the case if all the homogeneous components of $\iota_i(\calC)$ are isotropic for the trace form (recall $T(\iota_i(x)\iota_i(y))=4n(x,y)$ for any $x,y\in\calC$ and any $i=1,2,3$). We may assume $i=1$. There is an element $f_1\in\SOrt(\calC,n)$ which takes this basis to our ``good basis'' $\calB=\{e_1,e_2,u_1,u_2,u_3,v_1,v_2,v_3\}$ of $\calC$. Take $c\in\Spin(\calC,n)$ such that $f_1=\chi_c$ and consider the automorphism in $\Stab_{\Aut\calA}(E_1,E_2,E_3)$ determined by the related triple $(\chi_c,\rho^+_c,\rho^-_c)$ (see Corollary \ref{co:StabEis}).

Therefore we may assume, through this automorphism, that all the elements $\iota_1(e_j)$, $\iota_1(u_i)$ and $\iota_1(v_i)$, for $j=1,2$ and $i=1,2,3$, are homogeneous. Then
\[
\iota_1(v_1)\bigl(\iota_1(v_2)\bigl(\iota_1(v_3)\iota_3(\calC)\bigr)\bigr)=
\iota_2(((\calC v_3)v_2)v_1)=\bF\iota_2(e_1),
\]
and this proves, since $\iota_3(\calC)$ is a graded subspace, that $\iota_2(e_1)$ is homogeneous. In the same vein, we get that $\iota_2(e_2)$, $\iota_3(e_1)$ and $\iota_3(e_2)$ are homogeneous. Finally, $\iota_2(u_2)=-\iota_3(e_2)\iota_1(u_2)$ and $\iota_2(v_2)=-\iota_3(e_1)\iota_1(v_2)$ are homogeneous too.

Theorem \ref{th:CartanGrading} finishes the proof in this case.

\smallskip

Otherwise, in each $\iota_i(\calC)$ we may find some homogeneous element $\iota_i(x_i)$ with $n(x_i)\ne 0$, and we may scale it to get $n(x_i)=1$.

\begin{lemma}\label{le:iota12(1)}
Let $x_1,x_2\in \calC$ be elements of norm $1$, then there is an automorphism $\varphi\in\Stab_{\Aut\calA}(E_1,E_2,E_3)$ such that $\varphi(\iota_i(x_i))=\iota_i(1)$, for $i=1,2$.
\end{lemma}
\begin{proof}
First take an element $f_1\in\SOrt(\calC,n)$ which takes $x_1$ to $1$, and extend it as before to find a related triple $(f_1,f_2,f_3)$. The associated automorphism in $\Stab_{\Aut\calA}(E_1,E_2,E_3)$ takes $\iota_1(x_1)$ to $\iota_1(1)$ and $\iota_2(x_2)$ to some $\iota_2(y_2)$ with $n(y_2)=1$. Thus we may assume $x_1=1$.

Assuming $x_1=1$, take an element $a\in\calC_0$ with $n(a)=1$, $n(a,x_2)=0$. Then $n(x_2a,1)=n(x_2,\bar a)=-n(x_2,a)=0$, so $x_2a\in\calC_0$, and $n(x_2a)=n(x_2)n(a)=1$. Consider the element $c=(x_2a)\cdot a\in\Spin(\calC_0,n)$. Then $(\chi_c,\rho^+_c,\rho^-_c)$ is a related triple inducing an automorphism $\varphi$ in $\Stab_{\Aut\calA}(E_1,E_2,E_3)$ with $\varphi(\iota_1(1))=\iota_1(\chi_c(1))=\iota_1(1)$ and $\varphi(\iota_2(x_2))=\iota_2(\rho^+_c(x_2))=-\iota_2((x_2a)(x_2a))=\iota_2(1)$, as required.
\end{proof}

Therefore, in this situation we may assume that $\iota_1(1)$ and $\iota_2(1)$ are homogeneous elements. Let $a=\degree\iota_1(1)$ and $b=\degree\iota_2(1)$. Since $\iota_i(1)^2=4(E_{i+1}+E_{i+2})$ is an idempotent, we get $a^2=b^2=e$.

For $x,y\in\calC$, $\iota_3(xy)=\iota_1(\bar x)\iota_2(\bar y)=\bigl(\iota_2(1)\iota_3(x)\bigr)\bigl(\iota_3(y)\iota_1(1)\bigr)$, so if we define $\calC_g=\{x\in\calC: \iota_3(x)\in\calA_{abg}\}$ we get that for $x\in\calC_g$ and $y\in\calC_h$, $\iota_3(xy)\in (\calA_b\calA_{abg})(\calA_{abh}\calA_a)\subset \calA_{abgh}$, so $\calC_g\calC_h\subset \calC_{gh}$ and $\calC=\oplus_{g\in G}\calC_g$ is a $G$-grading on $\calC$.

Hence either there is a good basis of $\calC$ consisting of homogeneous elements, but then $\iota_3(\calC)$ has a basis consisting of homogeneous elements forming four orthogonal hyperbolic pairs, and this case has already been treated, or this grading in $\calC$ is equivalent to the $\bZ_2^3$-grading on $\calC$, and Theorem \ref{th:Z25Grading} shows that our grading $\Gamma$ is induced by the $\bZ_2^5$-grading of $\calA$.

\smallskip

In fact, we obtain more than what we need for the proof of Theorem \ref{th:Main}:

\begin{proposition}\label{pr:degree3strong}
Let $\Gamma: \calA=\bigoplus_{g\in G}\calA_g$ be a grading of the Albert algebra with $E_1,E_2,E_3\in\calA_e$. If there exists $i=1,2,3$ and an element $x\in\calC$ with $n(x)=0$ and $\iota_i(x)$ homogeneous, then $\Gamma$ is induced from the Cartan grading. Otherwise $\Gamma$ is induced from the $\bZ_2^5$-grading and all homogeneous components in each $\iota_j(\calC)$, $j=1,2,3$, are one-dimensional and orthogonal relative to the trace form.

Moreover, in the latter case, up to equivalence there are three different gradings whose universal grading groups and types are $\bZ_2^5$ and $(24,0,1)$, $\bZ_2^4$ and $(7,8,0,1)$, and $\bZ_2^3$ and $(0,0,7,0,0,1)$. The homogeneous component of highest dimension is $\cA_e$ in all cases.
\end{proposition}

\begin{proof}
If $\iota_i(x)$ is a nonzero homogeneous element with $n(x)=0$, then since the trace form is nondegenerate and $T(\iota_j(a)\iota_j(b))=4n(a,b)$ for any $j=1,2,3$ and $a,b\in\calC$, there is another homogeneous element $\iota_i(y)$ with $n(y)=0$ and $n(x,y)=1$. Then $n(x+y)=1$ so $\calC=(\bar x+\bar y)\calC=\bar x\calC + \bar y\calC$. As $\bar x\calC$ and $\bar y\calC$ are isotropic spaces, its dimension is at most $4$. We get $\calC=\bar x\calC\oplus\bar y\calC$, so $\iota_{i+2}(\calC)=\iota_{i+2}(\bar x\calC)\oplus\iota_{i+2}(\bar y\calC)=\iota_i(x)\iota_{i+1}(\calC)\oplus\iota_i(y)\iota_{i+1}(\calC)$ is the direct sum of two isotropic graded subspaces (for the trace form). Therefore, $\iota_{i+2}(\calC)$ has a basis consisting of homogeneous elements forming four orthogonal hyperbolic pairs, and hence $\Gamma$ is induced from the Cartan grading.

Otherwise all the homogeneous components in each graded subspace $\iota_j(\calC)$ are one dimensional and not isotropic, and hence orthogonal relative to the trace form, because of Theorem \ref{th:trace}. The arguments preceding this proposition show that we may assume $\degree\iota_1(1)=a$, $\degree\iota_2(1)=b$ and $\degree\iota_3(w_j)=abc_j$, $j=1,2,3$, with all the elements $a,b,c_1,c_2,c_3$ having order $2$, and that $\cC$ is graded with $\degree w_j=c_j$, $j=1,2,3$, so the subgroup $H$ generated by $c_1,c_2,c_3$ is isomorphic to $\bZ_2^3$. If $a,b\in H$, Lemma \ref{le:iota12(1)} allows us to assume $a=b=e$ and we get $\Supp\Gamma=H\cong\bZ_2^3$. If only one of $a$, $b$ or $ab$ are in $H$, by symmetry we may assume $a\in H$, and again we may assume $a=e$, thus getting $\Supp\Gamma=\langle b,H\rangle\cong\bZ_2^4$. Otherwise $\Supp\Gamma\cong\ZZ_2^5$, and $\Gamma$ is equivalent to the fine $\bZ_2^5$-grading. The types are easily computed.
\end{proof}

\subsection{Degree $2$}\label{ss:Degree2}

If the degree of $\calA_e$ is $2$, Corollary \ref{co:trace} shows that $\calA_e=\bF E\oplus\bF (1-E)$ for an idempotent $E$ with $T(E)=1$ (and hence $T(1-E)=2$). We may assume that $E=E_1$, so that $\wt{E}=1-E=E_2+E_3$. The grading on $\calA$ restricts to a grading on $\{X\in \calA\;|\; EX=0\}=\bF E_2\oplus\bF E_3\oplus\iota_1(\calC)=\bF \wt{E}\oplus \calV$, where $\calV=\bF (E_2-E_3)\oplus\iota_1(\calC)$, which is the Jordan algebra of a quadratic form with unity $\wt{E}$, because $(E_2-E_3)^2=E_2+E_3=\wt{E}$, $(E_2-E_3)\iota_1(\calC)=0$ and $\iota_1(x)\iota_1(y)=2n(x,y)\wt{E}=\frac{1}{2}T(\iota_1(x)\iota_1(y))\wt{E}$. Hence $XY=\frac{1}{2}T(XY)\wt{E}$ for any $X,Y\in\calV$. But the gradings on the Jordan algebras of quadratic forms are quite easy to describe: the unity is always in the identity component, and the restriction of the grading to the vector space $\calV$ is just a decomposition into subspaces: $\calV=\oplus_{g\in G}\calV_g$, with $T(\calV_g\calV_h)=0$ unless $gh=e$. Then either:
\begin{enumerate}
\item[1)] For any $g\in\Supp(\calV)$, $g^2=e$ and $\dim\calV_g=1$, or
\item[2)] There are homogeneous elements $X,Y\in\calV$ with $T(X^2)=T(Y^2)=0$ and $T(X,Y)=1$.
\end{enumerate}

Let us prove that the first case is not possible. Assume that for any $g\in\Supp(\calV)$, $g^2=e$ and $\dim\calV_g=1$. Let $H$ be the subgroup of $G$ generated by $\Supp(\calV)$, which is $2$-elementary: $H\cong \bZ_2^r$, with $r\geq 4$ as $\dim \calV=9$. Since $\{e\}\cup\Supp(\calV)$ has $10$ elements, it is not a subgroup of $H$, and hence there are elements $g\ne h\in\Supp(\calV)$ such that $gh\not\in\Supp(\calV)$. Then $\calV_g=\bF X$ for some $X$ with $X^2=\tilde E$. Hence $\wt{E}_2=\frac{1}{2}(\wt{E}+X)$ and $\wt{E}_3=\frac{1}{2}(\wt{E}-X)$ are nonzero orthogonal idempotents whose sum is $\wt{E}=1-E_1$. Thus $E_1$, $\wt{E}_2$ and $\wt{E}_3$ are orthogonal primitive idempotents and we may assume that $E_2=\frac{1}{2}(\wt{E}+X)$ and $E_3=\frac{1}{2}(\wt{E}-X)$, so that $X=E_2-E_3$. Then we have $\calV=\bF(E_2-E_3)\oplus\iota_1(\calC)$ and $g\not\in\Supp(\iota_1(\calC))$.

Let $\wb{G}=G/\langle g\rangle$ and consider the induced $\wb{G}$-grading on $\calA$, denoting by $\bar a$ the class of $a\in G$ modulo $\langle g\rangle$. Then $E_1,E_2,E_3\in\calA_{\bar e}$, so that each $\iota_i(\calC)$ are graded subspaces.

Besides, $\iota_1(\calC)$ is already a graded subspace of the original $G$-grading whose homogeneous components are all one-dimensional and non isotropic (relative to the norm of $\calC$). Moreover, since $\iota_1(\calC)_{gh}=\calV_{gh}=0$, $\iota_1(\calC)_{\bar h}=\iota_1(\calC)_{h}\oplus\iota_1(\calC)_{gh}=\iota_1(\calC)_h$ is one-dimensional and not isotropic. Proposition \ref{pr:degree3strong} gives that each homogeneous component of the $\wb{G}$-grading on each $\iota_i(\calC)$ is one-dimensional and not isotropic.

Take $a\in G$ such that $\iota_2(\calC)_{\bar a}\ne 0$, so that there is an element $x\in\calC$ with $n(x)\ne 0$ such that $\iota_2(\calC)_{\bar a}=\bF \iota_2(x)$. Then:
\[
\bigl(\iota_2(\calC)\oplus\iota_3(\calC)\bigr)_{\bar a}=
\iota_2(\calC)_{\bar a}\oplus\iota_3(\calC)_{\bar a}.
\]
If $\iota_3(\calC)_{\bar a}=0$, then $\iota_2(x)$ is homogeneous for the $G$-grading, and so is $\iota_2(x)^2=4n(x)(E_1+E_3)$, a contradiction with $\calA_e=\bF E_1\oplus\bF (E_2+E_3)$. Hence we have $\iota_3(\calC)_{\bar a}\ne 0$.

We conclude that the supports, for the $\wb{G}$-grading, of both $\iota_2(\calC)$ and $\iota_3(\calC)$ coincide. But since $n(x)\ne 0$, we have  $\iota_3(\calC)=\iota_1(\calC)\iota_2(\calC)_{\bar a}$. Since $\iota_3(\calC)_{\bar a}\ne 0$, it follows that $\iota_1(\calC)_{\bar e}\ne 0$, which means $\iota_1(\calC)_g\ne 0$, a contradiction with $g\not\in\Supp(\iota_1(\calC))$.

\smallskip

We are left with the second case, i.e., there are homogeneous elements $X\in\calV_g$, $Y\in\calV_{g^{-1}}$ with $T(X^2)=T(Y^2)=0$ and $T(XY)=1$, and $g\ne e$ because $\calA_e=\bF E\oplus\bF \wt{E}$. Then $(X+Y)^2=T(XY)\wt{E}=\wt{E}$ and hence $\frac12(\wt{E}-X-Y)$ and $\frac12(\wt{E}+X+Y)$ are nonzero idempotents with sum $\wt{E}$, so we may assume $X+Y=E_3-E_2$. Then $X-Y$ is an element of $\{Z\in\calA\;|\; E_1Z=0=(E_2-E_3)Z\}=\iota_1(\calC)$, and $T((X-Y)^2)=-2$. By Lemma \ref{le:iota12(1)}, we may assume $X-Y=\frac{\bi}{2}\iota_1(1)$. In other words, we may assume that the elements $S^+=X=(E_3-E_2)+\frac{\bi}{2}\iota_1(1)$ and $S^-=Y=(E_3-E_2)-\frac{\bi}{2}\iota_1(1)$ are homogeneous, say $S^+\in\calA_g$ and $S^-\in\calA_{g^{-1}}$ (because $S^+S^-=2\wt{E}\in\calA_e$).

Consider the $\bZ$-grading of $\calA$ in \eqref{eq:ZGrading}. The subspaces $\calA_{\pm 1}=\{Z\in\calA\;|\; EZ=\frac{1}{2}Z,\ S^{\pm}Z=0\}$ are then graded subspaces as well as $\calA_0=\bF E\oplus\bF \wt{E}\oplus \nu(\calC_0)$, since $\nu(\calC_0)=\{Z\in \calA\;|\; EZ=0=S^{\pm}Z\}$.

Assume now that there is an element $0\ne x\in\calC$ with $n(x)=0$ such that $\nu_+(x)$ is homogeneous: $\nu_+(x)\in(\calA_1)_{h_1}$. The nondegeneracy of the trace form shows that there is an homogeneous element $\nu_-(y)\in(\calA_{-1})_{h_1^{-1}}$ with $T(\nu_+(x)\nu_-(y))=8n(x,y)\ne 0$.  Then $\nu_+(x)\nu_-(y)=2n(x,y)(2E+\wt{E})+\nu(\bar xy-\bar yx)\in\calA_e=\bF E\oplus\bF\wt{E}$. Hence $\bar xy=\bar yx$. But then $n(x,y)1=\bar xy+\bar yx=2\bar xy$, a contradiction, since $n(\bar xy)=n(x)n(y)=0$ while $n(x,y)\ne 0$ and $n(1)=1\ne 0$.

Therefore, all the homogeneous components in $\calA_1$ are one-dimensional and not isotropic (relative to the norm of $\calC$ once we identify $\calA_1=\nu_+(\calC)$ with $\calC$). Fix an homogeneous element $\nu_+(x)\in(\calA_1)_a$, with $n(x)=1$. Then $\nu_+(x)^2=4n(x)S^+$, so $a^2=g$. The proof of Lemma \ref{le:iota12(1)} shows that there is an element $c\in\Spin(\calC_0,n)$ such that $\rho^+_c(x)=1$, so Remark \ref{re:StabEiiota1} allows us to assume that $x=1$. Thus we have $\nu_+(1)\in(\calA_1)_a$, $a^2=g$, and hence $\nu_-(1)=S^-\nu_+(1)\in(\calA_{-1})_{a^{-1}}$. In this situation, for any $x,y\in\calC$ such that $\nu_+(x)\in(\calA_1)_{h_1}$, $\nu_+(y)\in(\calA_1)_{h_2}$, we have:
\[
\begin{split}
(\nu_+(x)\nu_-(1))\nu_+(y)
    &=\Bigl(2n(x,1)(2E+\tilde E)+\nu(\bar x-x)\Bigr)\nu_+(y)\\
    &=3n(x,1)\nu_+(y)+\nu_+(y(\bar x-x))\\
    &=4n(x,1)\nu_+(y)-\nu_+(yx),\quad\text{as}\quad x+\bar x=n(x,1)1.
\end{split}
\]
If $n(x,1)\ne 0$, then $0\ne \nu_+(x)\nu_+(1)\in\bF S^+$, so that $h_1a=g=a^2$, so $h_1=a$ and $(\nu_+(x)\nu_-(1))\nu_+(y)\in(\calA_1)_{aa^{-1}h_2}=(\calA_1)_{h_2}$, and $\nu_+(yx)\in (\calA)_{a^{-1}h_1h_2}$. On the other hand, if $n(x,1)=0$, then $\nu_+(yx)=-(\nu_+(x)\nu_-(1))\nu_+(y)\in(\calA_1)_{a^{-1}h_1h_2}$ too.

Thus, consider the subspaces $\calC_h=\{x\in\calC\;|\; \nu_+(x)\in (\calA_1)_{ah}\}$ for $h\in G$. Then $\calC_{h_1}\calC_{h_2}\subset \calC_{h_1h_2}$ and we get a grading of $\calC$ in which all the homogeneous components are one-dimensional. Hence this is isomorphic to the $\bZ_2^3$-grading of $\calC$. Since $1\in\calC_e$, we have $\nu_+(1)\in\calA_a$, $\nu_-(1)\in\calA_{a^{-1}}$, and $\nu(w_j)=\nu_+(w_j)\nu_-(1)$ are homogeneous too, for $w_1$, $w_2$ and $w_3$ as in Theorem \ref{th:ZZ23Grading}. This Theorem shows that $\Gamma$ is induced from the $\bZ\times\bZ_2^3$-grading.

In fact, we can say more. Let $a=\degree\nu_+(1)$ and $b_j=\degree\nu(w_j)$, $j=1,2,3$. Then the subgroup $H=\langle b_1,b_2,b_3\rangle$ is isomorphic to $\bZ_2^3$ and $a^2=g\ne e$ as $\dim\cA_e=2$. Then $\Supp\Gamma=\langle a,H\rangle$, and the homogeneous components of the $5$-grading in \eqref{eq:ZGrading} have supports $\Supp\cA_{\pm 2}=\{a^{\pm 2}\}$, $\Supp\cA_{\pm 1}=a^{\pm 1}H$, $\Supp\cA_0=H$. If this subsets are disjoint, $\Gamma$ is equivalent to the $\bZ\times\bZ_2^3$-grading. Otherwise we have one of the following possibilities:
\begin{itemize}
\item $a^4=e$ but $a^2\not\in H$, thus getting a $\bZ_4\times\bZ_2^3$-grading of type $(23,2)$.
\item $a^2\in a^{-1}H$. In this case $a^3=b\in H$, and hence $(ab)^3=1$ and $(ab)^2=a^2$. As before we may change $a$ by $ab$ and hence assume $a^3=e$. We get a $\bZ_3\times\bZ_2^3$-grading of type $(21,3)$.
\item $a^2\in H$ (recall $a^2\ne e$). Since all the homogeneous components of the $\bZ_2^3$-grading of $\cC$, with the exception of the neutral component, play the same role we may assume $a^2=b_1$ and we obtain a unique, up to equivalence, grading by $\bZ_4\times \bZ_2^2$ of type $(6,9,1)$.
\end{itemize}

We summarize our arguments:

\begin{proposition}\label{pr:degree2strong}
Let $\Gamma: \cA=\oplus_{g\in G}\cA_g$ be a grading of the Albert algebra with $\dim\cA_e=2$. Then $\Gamma$ is induced from the $\bZ\times\bZ_2^3$-grading. Moreover, up to equivalence there are four such different gradings whose universal grading groups and types are $\bZ\times\bZ_2^3$ and $(25,1)$,
$\bZ_4\times\bZ_2^3$ and $(23,2)$, $\bZ_3\times \bZ_2^3$ and $(21,3)$, and $\bZ_4\times\bZ_2^2$ and $(6,9,1)$.\qed
\end{proposition}

\subsection{Degree $1$}\label{ss:Degree1}

Finally, consider the case of a grading $\Gamma: \calA=\bigoplus_{g\in G}\calA_g$ of the Albert algebra with $\dim\calA_e=1$, or $\calA_e=\bF 1$.

Let $g\in\Supp\Gamma$ be an element of order $2$. Let $\wb{G}=G/\langle g\rangle$ and consider the induced $\wb{G}$-grading. Then $\calA_{\bar e}=\calA_e\oplus\calA_g$ is a degree two Jordan algebra, so $\dim\calA_g=1$ by Corollary \ref{co:trace}, and $\calA_{\bar e}=\bF E\oplus\bF(1-E)$ for an idempotent $E$ with $T(E)=1$. But $\calA_g=\{X\in\calA_{\bar e}\;|\; X\not\in\bF 1,\ X^2\in\bF 1\}\cup \{0\}=\bF (1-2E)$, and $T(1-2E)=3-2=1$, while $T(\calA_g)=T(\calA_g\calA_e)=0$ by Theorem \ref{th:trace}, a contradiction. Therefore, for any element $g\in\Supp\Gamma$, we have $g=e$ or the order of $g$ is at least $3$.

Take now an element $g\in \Supp\Gamma$, $g\ne e$ (so its order is at least $3$), and take $X\in\calA_g$ and $Y\in\calA_{g^{-1}}$ with $T(XY)\ne 0$ Hence $0\ne XY\in \calA_e=\bF 1$ and we may take $XY=1$. This implies $T(1)\ne 0$, which shows that $\chr{\bF}\ne 3$.

The first linearization of equation \eqref{eq:generic3} gives
\[
X^2Y+2(XY)X-T(Y)X^2-2T(X)XY+S(X)Y+S(X,Y)X-N(X;Y)1=0
\]
($N(X;Y)$ being quadratic on $X$ and linear on $Y$). But $T(X)=T(X^2)=T(Y)=0$, so the component in $\calA_g$ of the above equation gives $X^2Y+2(XY)X+S(X,Y)X=0$, and $S(X,Y)=-T(XY)=-3$, so that $X^2Y+2X-3X=0$, or $X^2Y=X$. Then $X$ is invertible in the Jordan sense \cite[p.~51]{Jacobson} with inverse $Y$.  Since $T(X)=0=S(X)$, we have $X^3-N(X)1=0$, so $0\ne X^3\in\calA_e$, which forces $g^3=e$. Therefore, any element of $\Supp\Gamma$ different from $e$ has order $3$. Since we may assume that $G$ is generated by $\Supp\Gamma$, we conclude that $G$ is an elementary $3$-group.

Moreover, with $X$ as above, the quadratic operator $U_X$ is invertible and takes any $\calA_h$ to $\calA_{g^2h}$. In particular $\calA_e=U_X(\calA_g)$, which forces $\dim \calA_g=1$. Also, for any other $h\in\Supp\Gamma$, $U_X(\calA_h)=\calA_{g^2h}$, so we get that for any $g,h\in\Supp\Gamma$, $g^{-1}h\in\Supp\Gamma$. It follows that $\Supp\Gamma$ is a group, isomorphic to $\bZ_3^3$.

\smallskip

Since we have shown that $\chr{\FF}\ne 3$, the grading $\Gamma$ is given by three commuting order $3$ automorphisms $\varphi_1,\varphi_2,\varphi_3$ of $\calA$. Let $\calS$ be the subalgebra of elements fixed by $\varphi_1$ and $\varphi_2$. Then $\dim \calS=3$, $\calS=\calA_e\oplus\calA_g\oplus\calA_{g^2}$ for some $g\in \Supp\Gamma$. Take $X\in\calA_g$ with $X^3=1$. Thus $\calS$ is isomorphic to $\bF\times\bF\times \bF$, and we may assume that $\calS=\bF E_1\oplus\bF E_2\oplus\bF E_3$ with $\varphi_3(E_i)=E_{i+1}$ for any $i=1,2,3$.

For each $i$, the subspace $\iota_i(\calC)=\{X\in\calA\;|\; E_{i+1}X=\frac{1}{2}X=E_{i+2}X\}$ is invariant under $\varphi_1$ and $\varphi_2$, while $\varphi_3(\iota_i(\calC))=\iota_{i+1}(\calC)$.

For $x,y\in\calC$ define $x*y$ by $\iota_3(x*y)=\varphi_3(\iota_3(x))\varphi_3^2(\iota_3(y))$. Then:
\[
\begin{split}
\iota_3((x*y)*x)&=\varphi_3(\iota_3(x*y))\varphi_3^2(\iota_3(x))\\
    &=\bigl(\varphi_3^2(\iota_3(x))\iota_3(y)\bigr)\varphi_3^2(\iota_3(x))\\
    &=\varphi_3^2(\iota_3(x))\bigl(\varphi_3^2(\iota_3(x))\iota_3(y)\bigr).
\end{split}
\]
But $\varphi_3^2(\iota_3(x))=\iota_2(x')$ for some $x'\in\calC$ with $n(x)=n(x')$ (since $T(\iota_i(x)^2)=8n(x)$ and $T$ is invariant under $\varphi_3$), and
\[
\begin{split}
\iota_2(x')\bigl(\iota_2(x')\iota_3(y)\bigr)
    &= \iota_2(x')\iota_1(\bar x'\bar y)\\
    &=\iota_3(\wb{\bar x'\bar y}\bar x')\\
    &=\iota_3((yx')\bar x')=n(x')\iota_3(y)=n(x)\iota_3(y).
\end{split}
\]
Hence $(x*y)*x=n(x)y$ and, in the same vein, we get $x*(y*x)=n(x)y$. It follows that $(\calC,*)$ is a symmetric composition algebra (see \cite[Chapter VIII]{KMRT}), and $\varphi_1$ and $\varphi_2$ give, by restriction to $\iota_3(\calC)$, two commuting order $3$ automorphisms of $(\calC,*)$, and hence a grading of $(\calC,*)$ by $\bZ_3^2$. We obtain that $(\calC,*)$ is the Okubo algebra over $\bF$ and the grading is the unique, up to equivalence, $\bZ_3^2$-grading on $(\calC,*)$ \cite{ElduqueGrSym}.

Moreover, setting $\tilde\iota_i(x)=\varphi_3^i(\iota_3(x))$, we recover exactly the multiplication in $\calA$ in equations \eqref{eq:AlbertOkubo}. This shows that $\Gamma$ is equivalent to the $\bZ_3^3$-grading of $\calA$.

\smallskip

The proof of Theorem \ref{th:Main} is complete.

\subsection{Classification of $G$-gradings up to isomorphism}

Now we obtain, for any abelian group $G$, a classification of $G$-gradings on $\cA$ up to isomorphism.
We will need the following result describing the Weyl groups of the fine gradings $\Gamma_\cA^j$, $j=1,2,3,4$.

\begin{theorem}[\cite{EK_Weyl}]\label{th:Weyl_Albert}
Identifying $\supp\Gamma_\cA^1$ with the short roots of the root system $\Phi$ of type $F_4$, we have $\W(\Gamma_\cA^1)=\Aut\Phi$. $\W(\Gamma_\cA^2)$ is the stabilizer in $\Aut(\ZZ_2^2\times\ZZ_2^3)$ of the subgroup $\ZZ_2^3$ (as a set). $\W(\Gamma_\cA^3)=\Aut(\ZZ\times\ZZ_2^3)$. $\W(\Gamma_\cA^4)$ is the commutator subgroup of $\Aut(\ZZ_3^3)$.\qed
\end{theorem}

To state our classification theorem, we introduce the following notation:

$\bullet$ Let $\gamma=(b_1,b_2,b_3,b_4)$ be a quadruple of elements in $G$. Denote by $\Gamma^1_\cA(G,\gamma)$ the $G$-grading on $\cA$ induced from $\Gamma_\cA^1$ by the homomorphism $\ZZ^4\to G$ sending the $i$-th element of the standard basis of $\ZZ^4$ to $b_i$, $i=1,2,3,4$. For two such quadruples, $\gamma$ and $\gamma'$, we will write $\gamma\sim\gamma'$ if there exists $w\in\Aut\Phi$ such that $b'_j=b_1^{w_{1j}}b_2^{w_{2j}}b_3^{w_{3j}}b_4^{w_{4j}}$ where $w=(w_{ij})$ is considered as an element of $GL_4(\ZZ)$.

$\bullet$ Let $\gamma=(b_1,b_2,b_3)$ be a triple of elements in $G$ with $b_1b_2b_3=e$ and $b_i^2=e$, $i=1,2,3$. Let $H\subset G$ be a subgroup isomorphic to $\ZZ_2^3$. Fix an isomorphism $\alpha\colon\ZZ_2^3\to H$ and denote by $\Gamma^2_\cA(G,H,\gamma)$ the $G$-grading induced from $\Gamma^2_\cA$ by the homomorphism $\ZZ_2^2\times\ZZ_2^3\to G$ sending the $i$-th element of the standard basis of $\ZZ_2^2$ to $b_i$, $i=1,2$, and restricting to $\alpha$ on $\ZZ_2^3$. It follows from Theorem \ref{th:Weyl_Albert} that the isomorphism class of the induced grading does not depend on the choice of $\alpha$. For two such triples, $\gamma$ and $\gamma'$, we will write $\gamma\sim\gamma'$ if there exists $\pi\in\sg(3)$ such that $b'_i\equiv b_{\pi(i)}\pmod{H}$ for all $i=1,2,3$.

$\bullet$ Let $g$ be an element of $G$ such that $g^2\ne e$. Let $H\subset G$ be a subgroup isomorphic to $\ZZ_2^3$. Fix an isomorphism $\alpha\colon\ZZ_2^3\to H$ and denote by $\Gamma^3_\cA(G,H,g)$ the $G$-grading induced from $\Gamma^3_\cA$ by the homomorphism $\ZZ\times\ZZ_2^3\to G$ sending the element $1$ in $\ZZ$ to $g$ and restricting to $\alpha$ on $\ZZ_2^3$. It follows from Theorem \ref{th:Weyl_Albert} that the isomorphism class of the induced grading does not depend on the choice of $\alpha$. For two elements, $g$ and $g'$, we will write $g\sim g'$ if $g'\equiv g\pmod{H}$ or $g'\equiv g^{-1}\pmod{H}$.

$\bullet$ Let $H\subset G$ be a subgroup isomorphic to $\ZZ_3^3$. Then $\Gamma_\cA^4$ may be regarded as a $G$-grading with support $H$. Since $\W(\Gamma_\cA^4)$ has index $2$ in $\Aut(\ZZ_3^3)$, there are two isomorphism classes among the induced gradings ${}^\alpha\Gamma_\cA^4$ for various isomorphisms $\alpha\colon\ZZ_3^3\to H$. They can be distinguished as follows: fix a primitive third root of unity $\omega$ and a generating set $\{g_1,g_2,g_3\}$ for $H$, then in one isomorphism class we will have $(X_1 X_2)X_3=\omega X_1(X_2 X_3)$ and in the other $(X_1 X_2)X_3=\omega^{-1} X_1(X_2 X_3)$ where $X_i$ are nonzero elements with $\deg X_i=g_i$, $i=1,2,3$ --- see \cite[\S 4.5]{EK_Weyl}. We denote these two (isomorphism classes of) $G$-gradings by $\Gamma^4_\cA(G,H,\delta)$ where $\delta\in\{+,-\}$.

\begin{theorem}\label{th:Main_iso}
Let $\cA$ be the Albert algebra over an algebraically closed field of characteristic different from $2$. Let $G$ be an abelian group. Then any $G$-grading on $\cA$ is isomorphic to some $\Gamma^1_\cA(G,\gamma)$, $\Gamma^2_\cA(G,H,\gamma)$, $\Gamma^3_\cA(G,H,g)$ or $\Gamma_\cA^4(G,H,\delta)$ (characteristic $\ne 3$ in this latter case), but not two from this list. Also,
\begin{itemize}
\item $\Gamma^1_\cA(G,\gamma)$ is isomorphic to $\Gamma^1_\cA(G,\gamma')$ if and only if $\gamma\sim\gamma'$;

\item $\Gamma^2_\cA(G,H,\gamma)$ is isomorphic to $\Gamma^2_\cA(G,H',\gamma')$ if and only if $H=H'$ and $\gamma\sim\gamma'$;

\item $\Gamma^3_\cA(G,H,g)$ is isomorphic to $\Gamma^3_\cA(G,H',g')$ if and only if $H=H'$ and $g\sim g'$;

\item $\Gamma^4_\cA(G,H,\delta)$ is isomorphic to $\Gamma^4_\cA(G,H',\delta')$ if and only if $H=H'$ and $\delta=\delta'$.
\end{itemize}
\end{theorem}

\begin{proof}
By Theorem \ref{th:Main}, we know that any $G$-grading $\Gamma: \cA=\bigoplus_{g\in G}\cA_g$ is isomorphic to ${}^\alpha\Gamma^j_\cA$ for some $j=1,2,3,4$ ($j\ne 4$ if $\chr{\FF}=3$) and a homomorphism $\alpha\colon U(\Gamma^j_\cA)\rightarrow G$. In the case $j=2$, if the restriction $\alpha|_{\ZZ_2^3}$ is not one-to-one, then Propositions \ref{pr:degree3strong} tells us that $\Gamma$ can also be induced from $\Gamma^1_\cA$ by a homomorphism $\ZZ^4\to G$. In the case $j=3$, if the restriction $\alpha|_{\ZZ_2^3}$ is not one-to-one or $1\in\ZZ$ is sent to an element of order $\leq 2$, then Propositions \ref{pr:degree2strong} implies that the degree of the algebra $\cA_e$ is $3$ and hence, by Propositions \ref{pr:degree3strong}, $\Gamma$ is isomorphic to a grading induced from $\Gamma^1_\cA$ or $\Gamma^2_\cA$. In the case $j=4$, if $\alpha$ is not one-to-one, then $\cA_e$ has degree $3$ and the same argument applies. We have shown that $\Gamma$ is isomorphic to a grading from our list.

Now, two gradings on our list that have different $j$'s cannot be isomorphic, because the degree of $\cA_e$ is $1$ for $j=4$, it is $2$ for $j=3$, and $3$ for $j=1,2$; in the latter case the gradings can be distinguished as follows: for any grading induced from $\Gamma^1_\cA$ by a homomorphism $\ZZ^4\to G$ where $G$ is an elementary $2$-group, every homogeneous component $\cA_g$, $g\ne e$, has even dimension, whereas the gradings $\Gamma^2_\cA(G,H,\gamma)$ possess homogeneous components of odd dimension other than $\cA_e$ (see their types in Proposition \ref{pr:degree3strong}).

It remains to consider isomorphisms between two gradings with the same $j$. The ``if'' part follows from Theorem \ref{th:Weyl_Albert}, which shows that one grading can be mapped to the other by an automorphism in $\Aut(\Gamma^j_\cA)$. The proof of the ``only if'' part will be divided into cases according to the value of $j$.

1) Since $\Gamma^1_\cA$ is the eigenspace decomposition relative to a $4$-dimensional torus in $\Aut(\cA)$, and the latter is the simple algebraic group of type $F_4$, this case is covered by Proposition \ref{pr:isomorphism_Cartan}.

2) Suppose $\vphi\in\Aut(\cA)$ sends $\Gamma=\Gamma^2_\cA(G,H,\gamma)$ to $\Gamma'=\Gamma^2_\cA(G,H',\gamma')$. Then, in particular, it maps $\cA_e$ to $\cA'_e$. If $b_i\in H$ for all $i$, then $\Supp\Gamma=H$ and hence $\Supp\Gamma'=H$, which forces $H'=H$ and $b'_i\in H$ for all $i$. Suppose that at least one of the $b_i$ is not in $H$. Then, in fact, at least two of them, say $b_2$ and $b_3$, are not in $H$. Hence $\cA_e$ is not simple --- precisely, $\FF E_1$ is a factor of $\cA_e$. Then $\FF\vphi(E_1)$ is a factor of $\cA'_e$ and hence the idempotent $\vphi(E_1)$ is one of $E_i$, $i=1,2,3$. The automorphism of $\cA$ defined by $E_i\mapsto E_{i+1}$, $\iota_i(x)\mapsto \iota_{i+1}(x)$, for all $x\in\cC$ and $i=1,2,3$, belongs to $\Aut(\Gamma^2_\cA)$, so we may assume without loss of generality that $\vphi(E_1)=E_1$. It follows that $\vphi$ leaves the subspace $\FF E_2\oplus\FF E_3\oplus\iota_1(\cC)$ invariant. The support of this subspace is, on the one hand, $b_1 H$ and, on the other hand, $b'_1 H'$. It follows that $H=H'$ and $b_1\equiv b'_1\pmod{H}$. Also, $\vphi$ leaves the subspace $\iota_2(\cC)\oplus\iota_3(\cC)$ invariant, and the support of this subspace is, on the one hand, $b_2 H\cup b_3 H$ and, on the other hand, $b'_2 H\cup b'_3 H$. It follows that $b_2\equiv b'_2\pmod{H}$ and $b_3\equiv b'_3\pmod{H}$, or $b_2\equiv b'_3\pmod{H}$ and $b_3\equiv b'_2\pmod{H}$.

3) Suppose $\vphi\in\Aut(\cA)$ sends $\Gamma^3_\cA(G,H,g)$ to $\Gamma^3_\cA(G,H',g')$. Since $E=E_1$ is the unique idempotent of trace $1$ in $\cA_e$ and in $\cA'_e$, we have $\vphi(E_1)=E_1$. Hence the subspaces $\FF E_2\oplus\FF E_3\oplus\iota_1(\cC)$ and $\iota_2(\cC)\oplus\iota_3(\cC)$ are invariant under $\vphi$. Looking at the supports, we get:
\[
H\cup\{g^{\pm 2}\} = H'\cup\{(g')^{\pm 2}\}\quad\mbox{and}\quad gH\cup g^{-1}H =g'H'\cup (g')^{-1}H'.
\]
The first condition shows that the intersection $H\cap H'$ has at least $6$ elements, and hence it generates both $H$ and $H'$. Therefore, $H=H'$. Now the second condition  gives that $g'\equiv g\pmod{H}$ or $g'\equiv g^{-1}\pmod{H}$.

4) This case is clear from the definition of $\Gamma^4_\cA(G,H,\delta)$.
\end{proof}

\begin{corollary}\label{co:Main}
Let $\cA$ be the Albert algebra over an algebraically closed field of characteristic different from $2$. Then any abelian group grading on $\cA$ is either induced from the Cartan grading or is equivalent to one of the following:
\begin{itemize}
\item a $\bZ_2^5$-grading of type $(24,0,1)$, a $\bZ_2^4$-grading of type $(7,8,0,1)$, or a  $\bZ_2^3$-grading of type $(0,0,7,0,0,1)$, if the degree of the neutral component is $3$;

\item a $\bZ\times\bZ_2^3$-grading of type $(25,1)$,
a $\bZ_4\times\bZ_2^3$-grading of type $(23,2)$, a $\bZ_3\times \bZ_2^3$-grading of type $(21,3)$, or a $\bZ_4\times\bZ_2^2$-grading of type $(6,9,1)$, if the degree of the neutral component is $2$;

\item a $\bZ_3^3$-grading of type $(27)$ if the degree of the neutral component is $1$ and the characteristic is not $3$.\qed
\end{itemize}
\end{corollary}


\section{Gradings on $F_4$}\label{se:F4}

We continue to assume that the ground field $\FF$ is algebraically closed and $\chr{\FF}\ne 2$. The simple Lie algebra of type $F_4$ appears as the algebra of derivations of the Albert algebra. In order to describe it, consider first the local version of Definition \ref{df:relatedtriple}. Let $\cC$ be the Cayley algebra over $\bF$. Its \emph{triality Lie algebra} is defined as
\[
\tri(\cC)=\{(d_1,d_2,d_3)\in\frso(\cC,n)^3\;|\; d_1(x\bullet y)=d_2(x)\bullet y+x\bullet d_3(y)\ \forall x,y\in\cC\}.
\]
(Recall $x\bullet y=\bar x\bar y$ and $l_x(y)=r_y(x)=x\bullet y$.) As in Lemma \ref{le:relatedtriples}, if $(d_1,d_2,d_3)$ belongs to $\tri(\cC)$, so does $(d_3,d_1,d_2)$. The Lie bracket in $\tri(\cC)$ is the componentwise bracket, and we get the order $3$ automorphisms $\theta$ (triality automorphism):
\begin{equation}\label{eq:trialityauto}
\theta: (d_1,d_2,d_3)\mapsto (d_3,d_1,d_2).
\end{equation}
Each triple $(d_1,d_2,d_3)\in\tri(\cC)$ induces a derivation of the Albert algebra $\cA$:
\begin{equation}\label{eq:Dd1d2d3}
D_{(d_1,d_2,d_3)}: E_i\mapsto 0,\quad \iota_i(x)\mapsto \iota_i(d_i(x)),
\end{equation}
for any $i=1,2,3$ and $x\in \cC$. Also, for any $x\in\cC$ and $i=1,2,3$, consider the derivation $D_i(x)=2[L_{\iota_i(x)},L_{E_{i+1}}]$:
\begin{equation}\label{eq:Diotaix}
\begin{array}{lrll}
D_i(x):& E_i&\mapsto& 0,\quad E_{i+1}\;\mapsto\; \frac12\iota_i(x),\quad E_{i+2}\;\mapsto\; -\frac12 \iota_i(x),\\
 &\iota_i(y)&\mapsto& 2n(x,y)(-E_{i+1}+E_{i+2}),\\
 &\iota_{i+1}(y)&\mapsto& -\iota_{i+2}(x\bullet y),\\
 &\iota_{i+2}(y)&\mapsto& \iota_{i+1}(y\bullet x),
\end{array}
\end{equation}
for all $y\in\cC$. Then we get (\cite[Theorem IX.17]{Jacobson}):
\begin{equation}\label{eq:derA}
\Der(\cA)=D_{\tri(\cC)}\oplus\bigl(\bigoplus_{i=1}^3D_i(\cC)\bigr).
\end{equation}
One verifies at once the following properties (see \cite[\S 5.3]{ElduqueGrSym}):
\begin{equation}\label{eq:bracketF4}
\begin{split}
[D_{(d_1,d_2,d_3)},D_i(x)]&=D_i(d_i(x)),\\
[D_i(x),D_{i+1}(y)]&=D_{i+2}(x\bullet y),\\
[D_i(x),D_i(y)]&=2\theta^i(D_{x,y})
\end{split}
\end{equation}
for all $x,y\in\cC$, $(d_1,d_2,d_3)\in\tri(\cC)$ and $i=1,2,3$, where $\theta$ is the triality automorphism in \eqref{eq:trialityauto} and where $D_{x,y}=D_{t_{x,y}}$ with
\[
t_{x,y}=\bigl(\sigma_{x,y}, \frac12 n(x,y)\id-r_xl_y,\frac12 n(x,y)\id -l_xr_y\bigr)\in\tri(\cC),
\]
$\sigma_{x,y}(z)=n(x,z)y-n(y,z)x\in\frso(\cC,n)$. Moreover, the projection of $\tri(\cC)$ onto any of its components gives an isomorphism $\tri(\cC)\to \frso(\cC,n)$.

\smallskip

Take a ``good basis'' $\calB=\{e_1,e_2,u_1,u_2,u_3,v_1,v_2,v_3\}$ of $\cC$ and consider the subspace $\frh$ of $\frg=\Der(\cA)$ spanned by $D_{e_1,e_2}$ and $D_{u_i,v_i}$ for $i=1,2,3$. This is an abelian subalgebra of $\frg$. Actually, the image of $\frh$ in $\frso(\cC,n)$ under the projection of $\tri(\cC)$ onto its first component is the span of $\sigma_{e_1,e_2}$ and $\sigma_{u_i,v_i}$, $i=1,2,3$, so it is a Cartan subalgebra of $\frso(\cC,n)$.

Consider the linear maps $\epsilon_j\colon \frh\rightarrow \bF$, $j=0,1,2,3$, that constitute the dual basis to $D_{u_j,v_j}$, $j=0,1,2,3$, where $u_0\bydef e_1$ and $v_0\bydef e_2$.
Since we have:
\[
\begin{array}{rlll}
\sigma_{e_1,e_2}:& e_1\mapsto -e_1,& e_2\mapsto e_2,& u_i,v_i\mapsto 0,\\[4pt]
\sigma_{u_i,v_i}:& u_i\mapsto -u_i,& v_i\mapsto v_i,& e_1,e_2,u_j,v_j\mapsto 0\quad (j\ne i)\\[8pt]
\frac12\id -r_{e_1}l_{e_2}:& e_1\mapsto\frac12 e_1,& e_2\mapsto-\frac12 e_2,& u_i\mapsto -\frac12 u_i,\ v_i\mapsto \frac12 v_i,\\[8pt]
\frac12\id -r_{u_i}l_{v_i}:& e_1\mapsto\frac12 e_1,& e_2\mapsto-\frac12 e_2,& u_i\mapsto -\frac12 u_i,\ v_i\mapsto \frac12 v_i,\\[8pt]
& u_j\mapsto \frac12 u_j,& v_j\mapsto -\frac12 v_j& (j\ne i),\\[8pt]
\frac12\id -l_{e_1}r_{e_2}:& e_1\mapsto \frac12 e_1,& e_2\mapsto-\frac12 e_2,& u_i\mapsto \frac12 u_i,\ v_i\mapsto -\frac12 v_i,\\[8pt]
\frac12\id -l_{u_i}r_{v_i}:& e_1\mapsto-\frac12 e_1,& e_2\mapsto \frac12 e_2,& u_i\mapsto -\frac12 u_i,\ v_i\mapsto \frac12 v_i,\\[8pt]
& u_j\mapsto \frac12 u_j,& v_j\mapsto -\frac12 v_j& (j\ne i),
\end{array}
\]
we obtain that the weights of $\frh$ in $\iota_1(\cC)$, and hence the roots in $D_1(\cC)$, are $\pm\epsilon_j$, $j=0,1,2,3$, the weights in $\iota_2(\cC)$, and hence the roots in $D_2(\cC)$, are $\frac12(\pm\epsilon_0\pm\epsilon_1\pm\epsilon_2\pm\epsilon_3)$ with an even number of $+$ signs, and the weights in $\iota_3(\cC)$, and hence the roots in $D_3(\cC)$, are $\frac12(\pm\epsilon_0\pm\epsilon_1\pm\epsilon_2\pm\epsilon_3)$ with an odd number of $+$ signs. From $[\sigma_{a,b},\sigma_{x,y}]=\sigma_{\sigma_{a,b}(x),y}+\sigma_{x,\sigma_{a,b}(y)}$ for any $a,b,x,y\in\cC$ we obtain that the roots in $D_{\tri(\cC)}$ are $\pm\epsilon_r\pm\epsilon_s$, $0\leq r\ne s\leq 3$. Hence $\frh$ is a Cartan subalgebra of $\frg$ with the following set of roots:
\[
\Phi=\{\pm\epsilon_r\pm\epsilon_s\;|\; 0\leq r\ne s\leq 3\}\cup\{\pm\epsilon_r\;|\; 0\leq r\leq 3\}\cup\{\frac12(\pm\epsilon_0\pm\epsilon_1\pm\epsilon_2\pm\epsilon_3)\}.
\]
Note that the root spaces in $D_{\tri(\cC)}$ are the subspaces $\bF D_{u_i,u_j}$, $\bF D_{u_i,v_j}$ and $\bF D_{v_i,v_j}$ for $0\leq i\ne j\leq 3$, while in $D_i(\cC)$, $i=1,2,3$, the root spaces are the subspaces $\bF D_i(x)$ for $x\in\calB$. It follows at once that for any $\alpha\in\Phi$ and $X_\alpha\in\frg_\alpha$, the linear maps $X_\alpha^3$ on $\cA$, and $\ad_{X_\alpha}^3$ on $\frg$ are zero.

Consider the $\bZ^4$-grading on $\frg$ induced by the Cartan grading on $\cA$. Its homogeneous components are precisely the root spaces above, i.e., it is the Cartan decomposition of $\frg$ relative to $\frh$. We will call it the \emph{Cartan grading} on $\frg$.

\begin{proposition}\label{pr:innerderivations}
Let $\cA$ be the Albert algebra over an algebraically closed field of characteristic different from $2$ and let $\frg=\Der(\cA)$. Then any derivation of $\frg$ is inner.
\end{proposition}

\begin{proof}
This is well-known for $\chr{\FF}\ne 2,3$ (see \cite{Seligman}). We include a proof that is valid also in characteristic $3$, where the Killing form is trivial. The Cartan  grading on $\frg$ induces a grading on $\Der(\frg)$. It suffices to consider homogeneous elements $D\in\Der(\frg)$. Suppose $a\in\bZ^4$ and $D\in\Der(\frg)_a$. If $\frg_a=0$, then $D(\frh)=D(\frg_0)=0$. If $a=0$, then the subspaces $\frg_0$ and $\frg_\alpha$, $\alpha\in\Phi$, are invariant under $D$ and hence $D(\frh)=0$ again. Finally, suppose $\frg_a=\bF X_\alpha$ for some $\alpha\in\Phi$. Then there is a linear map $\lambda\colon \frh\rightarrow \bF$ such that $D(H)=\lambda(H)X_\alpha$ for all $H\in\frh$. Hence for any $H_1,H_2\in\frh$, we have $0=D([H_1,H_2])=[D(H_1),H_2]+[H_1,D(H_2)]$, which gives $\lambda(H_1)\alpha(H_2)=\lambda(H_2)\alpha(H_1)$. Therefore, either $\lambda=0$ or the linear maps $\lambda$ and $\alpha$ have the same kernel, so $\lambda=\mu\alpha$ for some $\mu\in\bF$. Hence the derivation $D+\mu\ad_{X_\alpha}$ annihilates $\frh$. We have shown that $\Der(\frg)=\ad(\frg) + \{D\in\Der(\frg)\;|\; D(\frh)=0\}$. Now take a system $\Delta$ of simple roots. For instance,
\begin{equation}\label{eq:simpleroots}
\Delta=\{\alpha_1,\alpha_2,\alpha_3,\alpha_4\}
\end{equation}
where $\alpha_1=\frac12(\epsilon_0-\epsilon_1-\epsilon_2-\epsilon_3)$, $\alpha_2=\epsilon_3$, $\alpha_3=\epsilon_2-\epsilon_3$ and $\alpha_4=\epsilon_1-\epsilon_2$. Any derivation $D\in\Der(\frg)$ which annihilates $\frh$ preserves the root spaces, so there are scalars $\mu_i\in\bF$ such that $D(X_{\alpha_i})=\mu_iX_{\alpha_i}$, and hence $D(X_{-\alpha_i})=-\mu_iX_{-\alpha_i}$. Take $H\in\frh$ such that $\alpha_i(H)=\mu_i$ for $i=1,2,3,4$. The multiplication rules in \eqref{eq:bracketF4} show that the elements $X_{\pm\alpha_i}$, $i=1,2,3,4$, generate $\frg$. It follows that $D=\ad_H$, which completes the proof.
\end{proof}

\begin{proposition}\label{pr:AlbertAd}
Let $\cA$ be the Albert algebra over an algebraically closed field of characteristic different from $2$ and let $\frg=\Der(\cA)$.
Then the map $\Ad\colon\Aut(\cA)\rightarrow\Aut(\frg)$, $\varphi\mapsto (D\mapsto \varphi\circ D\circ\varphi^{-1})$, is a group isomorphism.
\end{proposition}

\begin{proof}
Again, this is well-known for $\chr{\FF}\ne 2,3$ (see \cite[p.~71]{Seligman}). We include a proof that works also in characteristic $3$.
Since $\varphi\circ\ad_X\circ\varphi^{-1}=\ad_{\varphi(X)}$ for all $X\in\frg$, we see that $\Ad$ is one-to-one. The following argument will show that it is onto.

Consider the order $2$ automorphism of $\cC$ given by:
\begin{equation}\label{eq:sigma2C}
\sigma\colon  e_1\leftrightarrow e_2,\quad u_i\leftrightarrow v_i,\quad\mbox{for all}\quad i=1,2,3.
\end{equation}
This automorphism $\sigma$ extends to an order $2$ automorphism of $\cA$ by means of $\sigma(E_i)=E_i$, $\sigma(\iota_i(x))=\iota_i(\sigma(x))$, for all $i=1,2,3$ and $x\in\cC$, and hence it induces an order $2$ automorphism of $\frg$, which will be denoted by $\sigma$ as well. Note that the restriction of $\sigma$ to $\frh$ is $-\id$, and $\sigma$ takes any root space $\frg_\alpha$ to $\frg_{-\alpha}$.

Given $x,y,x',y'\in\cC$ with $n(x,x')=1=n(y,y')$ and $n(\bF x+\bF x',\bF y+\bF y')=0$, we get
\[
[[\sigma_{x,y},\sigma_{x',y'}],\sigma_{x,y}]
=[\sigma_{x,x'}+\sigma_{y,y'},\sigma_{x,y}]=-2\sigma_{x,y}.
\]
Hence, in particular, for $i\ne j$, we obtain: $[[D_{u_i,u_j},-\sigma(D_{u_i,u_j})],D_{u_i,u_j}]=2D_{u_i,u_j}$,
where, as before, $u_0=e_1$ and $v_0=e_2$. It follows that
\[
\{[D_{u_i,u_j},-\sigma(D_{u_i,u_j})],D_{u_i,u_j},-\sigma(D_{u_i,u_j})\}
\]
is an $\frsl_2$-triple in $\frg$, i.e., a triple $\{E,F,H\}$ satisfying $[H,E]=2E$, $[H,F]-2F$ and $[E,F]=H$, and thus spanning a subalgebra isomorphic to $\frsl_2(\bF)$. With the same arguments we get $\frsl_2$-triples starting with $D_{u_i,v_j}$ or $D_{v_i,v_j}$, $0\leq i\ne j\leq 3$. In a similar vein, for $x$ in the ``good basis'' $\calB$ of $\cC$:
\[
\begin{split}
[[D_i(x),D_i(\sigma(x))],D_i(x)]&=2[\theta^i(D_{x,\sigma(x)}),D_i(x)]\\
    &=2D_i\bigl(\sigma_{x,\sigma(x)}(x)\bigr)=-2D_i(x),
\end{split}
\]
so $\{[D_i(x),-\sigma(D_i(x))],D_i(x),-\sigma(D_i(x))\}$ is an $\frsl_2$-triple.

Take the system $\Delta$ of simple roots in \eqref{eq:simpleroots}, and the corresponding set of positive roots:
\[
\Phi^+=\{\epsilon_r,\epsilon_r\pm\epsilon_s,
\frac12(\epsilon_0\pm\epsilon_1\pm\epsilon_2\pm\epsilon_3)\;|\;0\leq r<s\leq 3\}.
\]
For each $\alpha\in\Phi^+$, choose the nonzero element $X_\alpha$ in the root spaces $\frg_\alpha$ to be of the form $D_{x,y}$ or $D_i(x)$ for some $x,y\in\calB$ and $i=1,2,3$. In particular,
\[
X_{\alpha_1}=D_3(e_1),\ X_{\alpha_2}=D_1(v_3),\ X_{\alpha_3}=D_{v_2,v_3},\ X_{\alpha_4}=D_{v_1,v_2}.
\]
Take $X_\alpha=-\sigma(X_{-\alpha})$ for $\alpha\in\Phi^-=-\Phi^+$.

With $H_i=[X_{\alpha_i},-\sigma(X_{\alpha_i})]$, the basis
\[
\cB_{Ch}=\{H_i,X_{\alpha}\;|\; 1\leq i\leq 4, \alpha\in\Phi\}
\]
is a Chevalley basis of $\frg$ (see \cite[proof of Proposition 25.2]{Humphreys}) whose structure constants lie in $\bZ$ if $\chr{\FF}=0$ and in the field $\bZ/p\bZ$ if $\chr{\bF}=p$. Moreover, the structure constants of the action of the elements $X_\alpha$ on $\cA$ are in $\frac12\bZ$ if $\chr{\bF}=0$ and in $\bZ/p\bZ$ if $\chr{\bF}=p$.

\smallskip

Let $\cA_\bC$ be the complex Albert algebra, so that $\frg_\bC=\Der(\cA_\bC)$ is the simple Lie algebra of type $F_4$ over $\bC$. Consider the ring $\bZ[\frac12]=\{\frac{a}{2^n}\;|\;a\in\bZ,\ n\in\bN\}$. In $\cA_\bC$, let $\cA_{\bZ[\frac12]}$ be the linear span of the basis $\{E_i,\iota_i(x)\;|\;i=1,2,3,\, x\in\calB\}$  over $\bZ[\frac12]$ and let $\frg_{\bZ[\frac12]}$ be the linear span of the basis $\cB_{Ch}$ over $\bZ[\frac12]$. Then our Albert algebra $\cA$ over $\bF$ is isomorphic to $\cA_{\bZ[\frac12]}\otimes_{\bZ[\frac12]}\bF$, and its Lie algebra of derivations $\frg=\Der(\cA)$ is isomorphic to $\frg_{\bZ[\frac12]}\otimes_{\bZ[\frac12]}\bF$.

According to Steinberg \cite[4.1]{Steinberg}, the automorphism group of $\frg$ is generated by the operators $\exp(\mu\,\ad_{X_\alpha})$, $\alpha\in\Phi$, $\mu\in\bF^\times$. These are indeed automorphisms, even in characteristic $3$, since they are obtained by specialization from the automorphism $\exp(t\,\ad_{X_\alpha})$ in $\frg_{\bZ[\frac12]}\otimes_{\bZ[\frac12]}\bZ[\frac12,t]$, which is a subalgebra of $\frg_{\bC}$ if we identify $t$ with a transcendental element in $\bC$. (Here we are using the same symbol $X_\alpha$ to denote an element in $\frg_\bC=\Der(\cA_\bC)$ and in $\frg=\Der(\cA)$, but this should cause no confusion.) Now,
\[
\exp(t\,\ad_{X_\alpha})(Y)=\exp (t X_{\alpha})Y\exp(-t X_{\alpha})
 =(\exp tX_{\alpha})Y(\exp tX_\alpha)^{-1},
\]
for all $Y\in\cA_\bC$, i.e., we have $\exp(t\,\ad_{X_\alpha})=\Ad{(\exp tX_\alpha)}$.

The operator $\exp\mu X_\alpha$ on $\cA$ is an automorphism of $\cA$, since it is obtained by specialization from an automorphism in $\cA_{\bZ[\frac12]}\otimes_{\bZ[\frac12]}\bZ[\frac12,t]$. We also have $\exp(\mu\,\ad_{X_{\alpha}})=\Ad{(\exp \mu X_\alpha)}$, which completes the proof.
\end{proof}

\begin{corollary}\label{co:AlbertAd}
Let $\cC$ be a Cayley algebra over a field $\bF$, $\chr{\FF}\ne 2$. Let $\cA=\sym_3(\cC)$ and $\frg=\Der(\cA)$.
Then $\Ad\colon\AAut(\cA)\rightarrow\AAut(\frg)$ is an isomorphism of affine group schemes.
\end{corollary}

\begin{proof}
Let $\bFalg$ be the algebraic closure of $\bF$. Since $\Der(\frg)\ot\bFalg=\Der_\bFalg(\frg\ot\bFalg)$, $\frg\ot\bFalg=\Der_\bFalg(\cA\ot\bFalg)$, and $\cA\ot\bFalg=\sym_3(\cC\ot\bFalg)$, we may pass from $\FF$ to $\bFalg$ and thus assume that $\FF$ is algebraically closed. Then $\Aut(\cA)$ is the simple algebraic group of type $F_4$ (see \cite[(25.13)]{KMRT} and the references therein) and hence
\[
\dim\Aut(\cA)=52=\dim\Der(\cA),
\]
which means that $\AAut(\cA)$ is smooth. Now, the maps $\Ad_\FF\colon\Aut(\cA)\to\Aut(\frg)$ and $\ad\colon\frg\to\Der(\frg)$ are both bijective, by Propositions \ref{pr:AlbertAd} and \ref{pr:innerderivations}, respectively. The result follows.
\end{proof}

Now Theorems \ref{transfer} and \ref{transfer_fine} yield the following result:

\begin{theorem}\label{th:transfer_F4}
Let $\cA$ be the Albert algebra over an algebraically closed field $\bF$, $\chr{\FF}\ne 2$. Then the abelian group gradings on $\Der(\cA)$ are those induced by such gradings on $\cA$. The algebras $\cA$ and $\Der(\cA)$ have the same classification of fine gradings up to equivalence and, for any abelian group $G$, the same classification of $G$-gradings up to isomorphism.\qed
\end{theorem}

\begin{corollary}\label{co:finegradingsF4}
We use the notation of Theorems \ref{th:transfer_F4} and \ref{th:Main}. Then, up to equivalence, the fine abelian group gradings on the simple Lie algebra $\frg=\Der(\cA)$, their universal groups and types are the following:
\begin{itemize}
\item The Cartan grading $\Gamma_\frg^1$ induced by $\Gamma^1_\cA$; universal group $\bZ^4$; type $(48,0,0,1)$.

\item The grading $\Gamma_\frg^2$ induced by $\Gamma_\cA^2$; universal group $\bZ_2^5$; type $(24,0,0,7)$.

\item The grading $\Gamma_\frg^3$ induced by $\Gamma_\cA^3$; universal group $\bZ\times \bZ_2^3$; type $(31,0,7)$.

\item The grading $\Gamma_\frg^4$ induced by $\Gamma_\cA^4$; universal group $\bZ_3^3$; type $(0,26)$ --- this one exists only if $\chr{\FF}\ne 3$.
\end{itemize}
\end{corollary}

\begin{proof}
Only the type of these gradings has to be checked and this is straightforward. The most difficult case is for the $\bZ\times\bZ_2^3$-grading. Since $\frg=[L_\cA,L_\cA]$ (see \cite[Corollary IX.11]{Jacobson}), we obtain: $\frg=\frg_{-3}\oplus \frg_{-2}\oplus\frg_{-1}\oplus\frg_0\oplus\frg_1\oplus\frg_2\oplus\frg_3$ where  $\frg_n=\sum_{r+s=n}[L_{\cA_r},L_{\cA_s}]$ and $\cA_r$ as in \eqref{eq:ZGrading}. But $[L_{S^\pm},L_{\nu_\pm(\cC)}]=0$, so $\frg_{\pm 3}=0$. The local version of Remark \ref{re:StabEiiota1} shows that $\frg_0$ contains a subalgebra isomorphic to $\frso(\cC_0,n)$. Also, $[L_E,L_{\cA_{\pm 1}}]=[L_E,L_{\nu_\pm(\cC)}]$ is an $8$-dimensional subspace of $\frg_{\pm 1}$, since $[L_E,L_{\nu_\pm(x)}](S^\mp)=\frac12 \nu_{\pm}(x)$, and $[L_{S^\pm},L_{\nu(\cC_0)}]$ is a $7$-dimensional subspace of $\frg_{\pm 2}$, since $[L_{S^\pm},L_{\nu(a)}](\nu_{\pm}(1))=-2\nu_{\pm}(a)$. It follows that $\dim\frg_0=22$, $\dim\frg_{\pm 1}=8$ and $\dim\frg_{\pm 2}=7$ (actually, $\frg_{\pm 1}=[L_E,L_{\nu_{\pm}(\cC)}]$ and $\frg_{\pm 2}=[L_{S^\pm},L_{\nu(\cC_0)}]$). Hence the type of the $\bZ\times\bZ_2^3$-grading, which is obtained by refining the $\bZ$-grading on $\frg$ above using the $\bZ_2^3$-grading on $\cC$, is $(31,0,7)$, where the seven $3$-dimensional homogeneous components are in $\frso(\cC_0,n)$, which is contained in $\frg_0$.
\end{proof}

Let $\Gamma^1_\frg(G,\gamma)$, $\Gamma^2_\frg(G,H,\gamma)$, $\Gamma^3_\frg(G,H,g)$, and $\Gamma^4_\frg(G,H,\delta)$ be the $G$-gradings induced by $\Gamma^j_\frg$, $j=1,2,3,4$, respectively, in the same way as for $\Gamma^j_\cA$ (see Theorem \ref{th:Main_iso}).

\begin{corollary}\label{co:gradings_F4_iso}
Let $\frg$ be the simple Lie algebra of type $F_4$ over an algebraically closed field $\FF$, $\chr{\FF}\ne 2$. Let $G$ be an abelian group. Then any $G$-grading on $\cA$ is isomorphic to some $\Gamma^1_\frg(G,\gamma)$, $\Gamma^2_\frg(G,H,\gamma)$, $\Gamma^3_\frg(G,H,g)$ or $\Gamma_\frg^4(G,H,\delta)$ (characteristic $\ne 3$ in this latter case), but not two from this list. Also,
\begin{itemize}
\item $\Gamma^1_\frg(G,\gamma)$ is isomorphic to $\Gamma^1_\frg(G,\gamma')$ if and only if $\gamma\sim\gamma'$;

\item $\Gamma^2_\frg(G,H,\gamma)$ is isomorphic to $\Gamma^2_\frg(G,H',\gamma')$ if and only if $H=H'$ and $\gamma\sim\gamma'$;

\item $\Gamma^3_\frg(G,H,g)$ is isomorphic to $\Gamma^3_\frg(G,H',g')$ if and only if $H=H'$ and $g\sim g'$;

\item $\Gamma^4_\frg(G,H,\delta)$ is isomorphic to $\Gamma^4_\frg(G,H',\delta')$ if and only if $H=H'$ and $\delta=\delta'$.\qed
\end{itemize}
\end{corollary}

\begin{corollary}\label{co:MainF4}
Using the notation of Corollary \ref{co:gradings_F4_iso}, any abelian group grading on $\frg$ is either induced from the Cartan grading or equivalent to one of the following:
\begin{itemize}
\item a $\bZ_2^5$-grading of type $(24,0,7)$, a $\bZ_2^4$-grading of type $(1,8,0,0,7)$, or a  $\bZ_2^3$-grading of type $(0,0,1,0,0,0,7)$;

\item a $\bZ\times\bZ_2^3$-grading of type $(31,0,7)$, a $\bZ_8\times\bZ_2^2$-grading of type $(19,6,7)$, a $\bZ_4\times\bZ_2^3$-grading of type $(17,7,7)$, a $\bZ_3\times \bZ_2^3$-grading of type $(3,14,7)$, or a $\bZ_4\times\bZ_2^2$-grading of type $(0,8,2,0,6)$;

\item a $\bZ_3^3$-grading of type $(0,26)$ if $\chr{\FF}\ne 3$.
\end{itemize}
\end{corollary}

\begin{proof}
Consider, for example, the gradings $\Gamma=\Gamma^2_\frg(G,H,g)$ and the corresponding grading on $\cA$. The homogeneous components of the $5$-grading in \eqref{eq:ZGrading} have supports $\Supp\cA_{\pm 2}=\{g^{\pm 2}\}$, $\Supp\cA_{\pm 1}=g^{\pm 1}H$, $\Supp\cA_0=H$. Hence $\Gamma$ has the following supports in each of the components of the $\ZZ$-grading $\frg=\bigoplus_{r=-2}^2\frg_r$:  $\Supp\frg_{\pm 2}=g^{\pm 2}(H\setminus \{e\})$ (as $\frg_{\pm 2}=[L_E,L_{\nu(\cC_0)}]$), $\Supp\frg_{\pm 1}=g^{\pm 1}H$, and $\Supp\frg_0=H$. If these subsets are disjoint, then $\Gamma$ is equivalent to the fine $\bZ\times\bZ_2^3$-grading. Otherwise we have several possibilities where some homogeneous components of this fine grading coalesce as in the arguments preceding Proposition \ref{pr:degree2strong}, plus a new possibility where $g^4\in H\setminus\{e\}$ and hence $\Supp\Gamma$ is a group isomorphic to $\bZ_8\times\bZ_2^2$.
With combinatorial arguments of this kind, one completes the proof.
\end{proof}


\end{document}